\documentclass[leqno]{amsart}%
\usepackage[foot]{amsaddr}
\usepackage{hyperref}
\usepackage{pdfsync}
\usepackage{dsfont}
\usepackage[dvipsnames]{xcolor}
\usepackage{amsthm}
\usepackage{enumerate}
\usepackage{amsfonts}
\usepackage{amssymb}
\usepackage{mathtools}
\usepackage{braket}
\usepackage{bm}
\usepackage{amsmath}
\usepackage{graphicx}
\usepackage{caption}
\usepackage{subcaption}
\usepackage{cleveref}
\usepackage{pb-diagram}
\numberwithin{equation}{section}
\newtheorem{theorem}{Theorem}[section]

\newtheorem{assumption}[theorem]{Assumption}
\newtheorem{corollary}[theorem]{Corollary}

\newtheorem{definition}[theorem]{Definition}
\newtheorem{lemma}[theorem]{Lemma}

\newtheorem{proposition}[theorem]{Proposition}
\theoremstyle{remark}
\newtheorem{remark}[theorem]{Remark}
\newtheorem{example}[theorem]{Example}

\newcommand{\R}{\mathbb{R}}
\newcommand{\N}{\mathbb{N}}
\renewcommand{\S}{\mathcal{S}}
\newcommand{\C}{\mathcal{C}}

\newcommand{\converges}[1]{ \overset{#1}{\longrightarrow}} 
\newcommand{\M}{\mathcal{M}}
\newcommand{\x}{\mathbf{x}}

\newcommand{\veps}{\varepsilon}
\newcommand{\eps}{\epsilon}

\DeclareMathOperator{\divergence}{div}

\DeclareMathOperator{\Prob}{\mathbb{P}}

\DeclareMathOperator{\Span}{span}
\DeclareMathOperator*{\esssup}{ess\,sup}

\newcommand{\red}{\color{red}}
\newcommand{\blue}{\color{blue}}

\definecolor{mygreen}{rgb}{0.1,0.75,0.2}

\newcommand{\nc}{\normalcolor}

\newcommand{\Ind}{\Theta}

\newcommand{\Sop}{\Lambda}

\newcommand{\W}{\mathbf{W}}
\newcommand{\D}{\mathbf{D}}

\newcommand{\PN}{\Pi_N}

\usepackage[colorinlistoftodos,prependcaption,textsize=tiny]{todonotes}

\title{Geometric structure of graph Laplacian embeddings}
\author{Nicol{\'a}s Garc{\'i}a Trillos}
\address{Department of Statistics, University of Wisconsin-Madison, Madison, WI}
\email{nicolasgarcia@stat.wisc.edu}
\author{Franca Hoffmann \and Bamdad Hosseini}
\address{Computing and Mathematical Sciences, Caltech, Pasadena, CA}
\email{fkoh@caltech.edu, bamdadh@caltech.edu}

\begin{document}
\maketitle

\begin{abstract}
  We analyze the spectral clustering procedure for identifying coarse structure in a data set $\x_1, \dots, \x_n$, and in particular study the geometry of graph Laplacian embeddings which form the basis for
  spectral clustering algorithms. More precisely, we assume that the data is sampled from a mixture model supported on a manifold $\M$ embedded in $\R^d$, and pick a connectivity length-scale $\veps>0$ to construct a kernelized graph Laplacian. We introduce a notion of a well-separated mixture model which only depends on the model itself, and prove that when the model is well separated, with high probability the embedded data set concentrates on cones that are centered around orthogonal vectors. Our results are meaningful in the regime where $\veps = \veps(n)$ is allowed to decay to zero at a slow enough rate as the number of data points grows. This rate depends on the intrinsic dimension of the
 manifold on which the data is supported.
\end{abstract}
\section{Introduction}

The goal of this article is to analyze the  \textit{spectral clustering} procedure for  
identifying coarse structure in a dataset by means of appropriate spectral embeddings. In particular, we combine ideas from spectral geometry, 
metastability, optimal transport, and  spectral analysis of elliptic operators to describe the geometry of the
embeddings that form the basis of spectral clustering in the case where the generative model for the data is sufficiently ``well-separated".

Throughout this article we assume the data set $\M_n= \{ \x_1, \dots, \x_n\}$ consists of points that 
are distributed according to an underlying probability measure $\nu$ supported on 
a manifold $\M$ embedded in some Euclidean space $\R^d$. The first step in 
spectral clustering is to construct a similarity graph $G:=\{ \M_n, \W_n\}$ 
with vertices $\M_n$ and edge weight matrix $\W_n$. The entries in $\W_n$ 
indicate how similar two points in $\M_n$ are; typical constructions are proximity graphs, where a length scale $\veps>0$ is chosen 
and high weights are given to pairs of points that are within distance $\veps$ from each other. Once the graph has been constructed a {\it graph Laplacian} operator is introduced.  The first few eigenvectors of the {graph Laplacian} are then used to define an embedding that maps the data points $\M_n$ into a low-dimensional Euclidean space. More precisely, the low-lying spectrum of the graph Laplacian induces an embedding $F_n : \left\{\x_1, \dots, \x_n \right\} \rightarrow \R^N$ mapping a data point $\x_i$ into a vector $F_n(\x_i)$ in $\R^N$ whose entries are the first $N$ eigenfunctions of the graph Laplacian evaluated at $\x_i$. Here we refer to the map $F_n$ as the {\it Laplacian embedding}.
After the embedding step, an algorithm such as
$k$-means is typically used to identify the meaningful clusters (see \cite{vonLux_tutorial}). While $k$-means is a standard choice in the literature, in the rest of the paper we are agnostic to the algorithm used to cluster the embedded data set. We refer the reader to section \ref{sec:Literature} and in particular to \cite{LingStromer} for alternatives to $k$-means.

How can one motivate the use of the previous spectral clustering algorithm? Consider the simple scenario where the manifold $\M$ has $N$ connected components. Then, given data points uniformly distributed on $\M$ one can construct the proximity graph $G$ (at least for $n$ large enough) in such a way that the associated graph Laplacian also has $N$ connected components and its first $N$ eigenvectors coincide with rescaled versions of the indicator functions of the $N$ connected components of
$\M$ restricted to $\M_n$. The resulting Laplacian embedding map $F_n$ sends the original data set into a set of $N$ orthogonal vectors in $\R^N$. Such geometric structure of the embedded dataset allows one to readily identify the ``true" coarse structure of the original dataset. The key question that we attempt to answer in this work is: what happens in more general settings where in particular $\M$ is not connected? We attempt to answer this question by:
\begin{enumerate}
\item Relaxing the notion of $N$ connected components of $\M$, and replacing it with a notion of well-separated mixture model $\nu $. We highlight that our notion of well-separatedness only depends on the mixture model $\nu$  and not on extrinsic objects like a user chosen kernel in the construction of the graph $G$.
\item Describing the geometry of the Laplacian embedding in terms of what we will call an \textit{orthogonal cone structure}, a generalization of the orthogonality condition of the embedded dataset observed when $\M$ has $N$ connected components. We consider this notion for arbitrary probability measures and not just for point clouds, thus extending the notion introduced in \cite{BinYu}. 
\end{enumerate}
The analysis that we present in this work is divided into two parts. First, we study the geometry of a continuum analogue Laplacian embedding using techniques from
PDEs and the calculus of variations. In the second part, we compare the two embeddings, discrete and continuum. For the second part we use recent results quantifying the approximation error of the spectra of continuum Laplacian operators with the spectra of graph Laplacians. It is only in the second step that we study statistical properties of random objects, and the first part of the paper is, to some extent, of independent interest.

Overall, our ideas, definitions, and results provide new theoretical insights on spectral clustering. They further justify the use of spectral clustering beyond the trivial setting where a manifold is disconnected. In particular,
we show that when the ground-truth measure $\nu$ has well-separated components, the geometry of the embedded data set $F_n(\M_n)$ is simple enough so that an algorithm like $k$-means is successful in correctly identifying the coarse structure of the data set. 

\subsection{Overview of main results}

Let us make our setup more precise. We assume that  $\M$ is a connected manifold and that $\nu$ is a mixture model with $N$ components supported on $\M$ (see Section \ref{Set-up} for detailed assumptions
on this model). We work in a non-parametric setting and assume that the components of $\nu$ have sufficiently regular densities with respect to the volume form of $\M$.  We will primarily work with graph Laplacian operators corresponding to \textit{kernelized} Laplacians (see \eqref{KernelizedWeights} for the definition and Section \ref{sec:differentweights} for some reasons behind this choice, where in particular we discuss that the kernelized Laplacian already covers many other cases of interest).
Before describing the geometry of the embedded data set $F_{n} (\M_n)$ we study a continuum limit analogue of spectral clustering which we refer to as \textit{measure clustering}. This can be thought as an ideal clustering problem where an infinite amount of data is available. Studying this continuum analogue  is motivated by recent results concerning the convergence of graph Laplacian operators to differential operators in the large data limit as
$n$, the number of vertices,  goes to infinity and  $\veps$, the {\it connectivity length scale} specifying
$G$, goes to zero sufficiently slowly.
The low-lying spectrum of the limiting differential operator induces an embedding 
$F: \M \rightarrow \R^N$  in a manner similar to $F_n$, but now at the continuum level. We characterize the geometry of the measure $F_\sharp \nu$ (i.e. the push-forward of $\nu$ through
the continuum limit Laplacian embedding $F$) which forms the basis of the measure clustering
procedure. In particular, we identify conditions on the ground-truth model that ensure that the measure $F_\sharp \nu$ has an orthogonal cone structure. This is the content of
our first key result Theorem \ref{mainTheorem} which is informally stated below.  
\begin{theorem}[Informal]\label{thm:well-separated-gives-cone-structure-continuum}
If the mixture $\nu$ has well-separated components (see Definition~\ref{def:well-separated-mixture}),
then  $F_\sharp\nu$, the push forward
of $\nu$ through the Laplacian embedding $F$, has an orthogonal cone structure (see Definition \ref{def:contcones}). 
\end{theorem}

Let us now turn our attention to  the geometry of the embedded point cloud $F_n( \M_n)$.
As mentioned above,
when constructing the graph $G$ from the data $\M_n$, one often scales  $\veps$ with
$n$ in such a way that when $n \rightarrow \infty$, $\veps \to 0$ (i.e. the weight matrix
$\W_n$ localizes) at a slow enough rate.  Analyzing this regime is motivated by the desire to balance between graph connectivity and computational efficiency.  Recent results in the literature provide explicit rates of convergence for the spectra of graph Laplacians towards the spectrum of an elliptic differential operator (see \eqref{ContLaplaciansRho} and Section~\ref{App:A} below) in such regime. Moreover, the convergence of the eigenvectors is with respect to a strong enough topology, implying that the empirical measure of $F_n(\M_n)$ and the measure $F_{\sharp} \nu$ get closer in the
Wasserstein sense as $n$ grows.  In turn, this convergence guarantees that if $F_{\sharp} \nu$ has an orthogonal cone structure, then $F_n(\M_n)$ will have one as well (i.e., orthogonal cone structures are stable under Wasserstein perturbations). In Theorem \ref{PropMuMun} we make the previous discussion mathematically precise. For the moment we provide an informal statement of that result.

\begin{theorem}[Informal]\label{thm:discrete-is-consistent-with-continuum}
Under suitable conditions, if $F_\sharp \nu$
has an orthogonal cone structure then, with very high-probability,
the empirical measure of the embedded cloud $F_n(\M_n)$ also has 
an orthogonal cone structure.
\end{theorem}

\subsection{Related work} 

\label{sec:Literature}

In recent years, graph-based methods  have been used extensively in machine learning literature, due to their ability to capture the intrinsic geometry of a data set with relatively simple algorithms. 
Operators and functionals defined on graphs often induce sensible approaches to address fundamental
data analysis tasks such as semi-supervised learning and clustering. 

In this paper we focus  on clustering. More specifically, clustering by means of appropriate embeddings using the spectrum of a graph Laplacian operator. In the context of stochastic block models, large sample properties of the spectra of graph Laplacians have been studied in \cite{Rohe}. In the geometric graph setting (our
setting) \cite{vonLuxburg} was the first to rigorously establish a statistical consistency result for  graph-based spectral embeddings when the number of data points $n$ goes to infinity and the connectivity $\veps>0$ remains constant.  Up until then most papers focused on proving pointwise convergence estimates (see \cite{bel_niy_LB,HeinvonLuxburgAudibert} and references within).  The work \cite{vonLuxburg} however, did not address the question of determining scalings for $\veps:=\veps(n) \rightarrow 0$ that implied spectral convergence towards differential operators in the
continuum limit. Moreover, it was still a pending task to analyze the entire spectral clustering algorithm, where in addition to the embedding step, there is an extra clustering step (typically $k$-means).

In \cite{GarciaTrillosSlepcev2015}, the authors introduced a framework using notions from optimal transport and the calculus of variations to study large data limits of optimization problems defined on random geometric graphs. Many large data limits of relevant functionals for machine learning have been studied using similar approaches \cite{SlepcevThorpe,StuartSlepcevThorpeDunlop,GarciaMurray,GarciaSanzAlonso}. In particular, the optimal transport framework was used in \cite{GarciaTrillosSlepcev2015} to analyze the entire spectral clustering procedure (embedding step and $k$-means step) and the authors were able to show statistical consistency for scalings $ \log(n)^{p_m}/n^{1/m} \ll \veps\ll 1$ where
$m$ is the intrinsic dimension of the manifold $\M$  and $p_m >0$ is
an appropriate constant depending on $m$. Nevertheless, the spectral convergence was asymptotic and no rate was provided until  \cite{Shi2015}. In \cite{BIK}, the authors took a somewhat similar approach to \cite{GTSSpectralClustering}, and in a non-probabilistic setting, conducted a careful analysis which ultimately allowed to provide very precise rates of convergence for the spectrum of graph Laplacians. The work \cite{Moritz} extended this analysis to cover the relevant probabilistic setting where point clouds are samples from a ground truth measure supported on a manifold. The rates in \cite{Moritz} are, to the best of our knowledge, state of the art. Many of the results and
ideas  introduced in \cite{Moritz} will be revisited in Section \ref{App:A}; as a matter of fact, the probabilistic estimates that we need for this paper are proved in a very similar fashion to \cite{Moritz} after making some slight modifications that we will discuss later on. 

All of the publications listed above focus on spectral convergence of graph Laplacians, but do not study the geometry of Laplacian embeddings themselves (i.e. what is the geometry of the embedded data set, and how such geometry may affect the subsequent clustering step). The work \cite{BinYu} formulates this problem in a precise mathematical setting, by introducing a notion of a well-separated data generating mixture model. Roughly speaking, the authors obtain a result that, in spirit, reads similarly to Theorem \ref{thm:discrete-is-consistent-with-continuum}. However, the notion of a well-separated model introduced in that work depends on extrinsic quantities such as the kernel used to construct the graph at the finite data level, and in particular they work in the regime where $n$ is thought to be large, but the connectivity length-scale $\veps$ does not decay to zero. The article \cite{BinYu} has inspired and motivated our work. 

Different scenarios where one can define a notion of well-separated components were studied in \cite{LittleMaggioniMurphy,LingStromer}. Both of these papers are closely related to our work, and we believe are complementary of each other. In both works, the respective notion of ``well-separated components" is used to help validate proposed methodologies for clustering. In \cite{LittleMaggioniMurphy}, the focus is on understanding the stability of spectral embeddings in situations where the original data set has a geometric structure that is highly anisotropic. The authors propose the use of a general family of distances (other than the Euclidean one) to construct graphs (which ultimately induce corresponding spectral embeddings). The model considered in that paper is more geometric in nature than the one we consider here. In \cite{LingStromer}, the authors introduce a notion of well-separated components at the finite data level and use it to validate a new algorithm for clustering which is based on a convex relaxation for a multi-way cut problem and does not require a $k$-means step, making it quite appealing. 

We believe it is worth exploring and understanding the connections of the many different mathematical ideas and models presented in \cite{LittleMaggioniMurphy,LingStromer} with the ones we present here.

\subsection{Outline}

The rest of the paper is divided as follows. Section \ref{Set-up} introduces the precise set-up of the paper and presents the main results. In particular, in Section \ref{sec:wellsepa} we introduce our notion of well-separated mixture models; our two main results, Theorem \ref{mainTheorem} on the geometry of continuum Laplacian embeddings (i.e. the measure clustering embedding) and Theorem \ref{PropMuMun} describing the geometry of graph Laplacian embeddings, are presented in Sections \ref{sec:continuum} and \ref{sec:fromconttodiscr} respectively.

Section \ref{sec:discussion} is devoted to discussing our notion of well-separated mixture models.
We present some illustrative examples, discuss the implications of our results, and finally comment on possible extensions and generalizations.

Proofs of the main results are postponed to Section \ref{sec:Proofs}. In particular, Section \ref{sec:Conic} focuses on stability properties of orthogonal cone structures, allowing us to prove Theorem \ref{mainTheorem} in Section \ref{sec:ProofsMain} after showing several preliminary results. In Section \ref{App:A} we make the connection between the eigenspaces generated by the first $N$ eigenvectors of the graph Laplacian and those of the continuum Laplacian. We conclude with Section \ref{sec:proof2}, where we present the proof of Theorem \ref{PropMuMun}.

\section{Setup and main results}
\label{Set-up}

Let $\mathcal{M}$ be a smooth, connected, orientable,  $m$-dimensional manifold embedded in $\R^d$. For the moment we think of $\M$ as either a compact manifold without boundary or the whole of $\R^d$. Let $\rho_1, \dots, \rho_N$ be a family of $C^1(\M)$ probability density functions on $\M$ (densities with respect to the volume form of $\M$) 
%satisfying
%\[\rho_i(x)>0 , \quad \forall x \in \M, 
%\forall i=1, \dots, N,\]
and let  $w_1,\dots, w_N$ be positive weights adding to one. The densities $\rho_1, \dots, \rho_N$ are weighted to produce the \textit{mixture model}
\begin{equation}
\rho(x) := \sum_{k=1}^{N} w_k \rho_k(x) , \quad  x \in \M.
\label{DefRho}
\end{equation}
We assume that $\rho(x)>0$ for all $x \in \M$  and let $\nu$ be the probability measure on $\M$ induced by $\rho$, that is, 
\[d \nu(x) = \rho(x) dx \in \mathcal{P}(\M),\]
where in the above and in the remainder $dx$ denotes integration with respect to $\M$'s volume form
and $\mathcal{P}(\M)$ denotes the space of complete Borel probability measures on $\M$.

Associated to each of the densities $\rho_k$ we introduce a differential operator $\Delta_{\rho_k}$, which for smooth functions $u$ is defined by
\begin{equation}
\Delta_{\rho_k} (u) := -\frac{1}{\rho_k}\divergence(\rho_k \nabla u) = -\Delta u -  \frac{1}{\rho_k} \nabla  \rho_k \cdot \nabla u.
\label{ContLaplacians}
\end{equation}
In the above, $\nabla $ stands for the gradient operator on $\M$, $\divergence $ is the divergence on $\M$, and $\Delta$ is the Laplace-Beltrami operator on $\M$. 
The term $\nabla  \rho_k \cdot \nabla u $ evaluated at a point $x \in \M$, is the inner product of the vectors $\nabla  \rho_k(x), \nabla u (x) $ in the tangent space $\mathcal{T}_x\M$; 
seen as vectors in $\R^d$, $\nabla  \rho_k(x)\cdot \nabla u (x)$ is just their inner product in $\R^d$. Likewise, we consider the operator,
\begin{equation}
\Delta_\rho (u) := -\frac{1}{\rho}\divergence(\rho \nabla u) = -\Delta u - \frac{1}{\rho}\nabla \rho \cdot \nabla u.
\label{ContLaplaciansRho}
\end{equation}
Given a density $\varrho$ on $\M$ we define the weighted function spaces
\begin{equation*}
  L^2(\M, \varrho):= \left\{ u: \M \mapsto \mathbb R \Big| \langle u, u \rangle_{\varrho} < + \infty  \right\},
\end{equation*}
equipped with the weighted inner product
\begin{equation*}
 \langle u, v \rangle_{\varrho} := \int_\M u(x) v(x) \varrho(x) dx.
\end{equation*}
Throughout the article we routinely write $L^2(\varrho)$ instead of $L^2(\M, \varrho)$ when
the domain of the function space is clear from the context. In
section~\ref{App:A} we modify our notation slightly and 
use $L^2(\mu)$ to denote $L^2(\M, \varrho)$ given the measure $\mu = \varrho dx$ and use
$\langle\cdot, \cdot\rangle_{L^2(\mu)}$ to denote $\langle \cdot, \cdot \rangle_\rho$. The reason for choosing the specific form \eqref{ContLaplaciansRho} for the operator $\Delta_\rho$  is justified by its close connection with the kernelized graph Laplacian that we will introduce and discuss in Sections \ref{sec:fromconttodiscr} and \ref{sec:differentweights} below.

We will make the following assumption.
\begin{assumption}
	\label{Assump1}
We assume that the mixture model $\nu$ satisfies the following: $\rho_1, \dots, \rho_N
\in C^1(\M)$ with $\int_\M\rho_k(x)\,dx=1$ for all $k$ are such that
the operators $\Delta_{\rho_k}$ and $\Delta_{\rho}$ have a discrete point spectrum with an associated orthonormal basis of eigenfunctions for $L^2(\rho_k)$  and $L^2(\rho)$ respectively.	
\end{assumption}

%
%This can be established by showing that  $\Delta_{\rho_k}+Id$ and $\Delta_{\rho}+Id$ are positive-definite and compact
%operators, which holds true if $\M$ is compact \red or if $\M$ is $\R^d$ and the densities $\rho_k$ are Gaussians.\nc \fhremark{In whole space, not sure $\Delta_{\rho_k}+I$ is compact. What about taking $\R^d$ case out?}
%\bhremark{I looked into this a bit more and I am not
%  convinced this is true. I haven't found a good
%reference either.}

% ----------------------------------------------------------------------------
\subsection{Well-separated mixture models}
\label{sec:wellsepa}

Let us now introduce the notion of a well-separated mixture model. This is a notion defined in terms of three quantities associated to the model $\nu$ that we refer to as \textit{overlapping}, \textit{coupling}, and \textit{indivisibility} parameters.

For $k =1, \dots, N$, define the functions
\begin{equation}\label{qk-definition}
 q_k := \sqrt{\frac{w_k \rho_k}{\rho}}.
\end{equation} 
The \textit{overlapping} parameter of the mixture model $\nu$ is the quantity 
%\begin{equation}
%\mathcal{S}_1 :=  \max_{i \not = j} \left  \langle \frac{q_i}{\sqrt{w_i}} , \frac{q_j}{\sqrt{w_j}} \right \rangle_\rho     =  \max_{i \not = j}  \int_{\M}\sqrt{\rho_i} \sqrt{\rho_j} dx,
%\label{eqn:S1} 
%\end{equation}
%and
\begin{equation}
\S :=   \max_{i \not = j}  \left \langle \left(\frac{q_i}{\sqrt{w_i}}\right)^2 , \left(\frac{q_j}{\sqrt{w_j}} \right)^2\right\rangle_\rho    =  \max_{i \not = j}   \int_{\M} \frac{\rho_i \rho_j}{\rho} dx,
\label{eqn:S2}
\end{equation}
% where $\langle \cdot, \cdot\rangle_{\rho}$ denotes the inner product induced by the measure $d \nu= \rho(x) dx $, that is,
% \[ \langle f , g \rangle_\rho := \int_{\M}fg d \nu = \int_{\M}fg \rho dx, \quad f , g \in L^2(\rho).\] 
Note that $\S\in\left(0,\frac{1}{\min_j w_j}\right)$ since for any $k\in\{1,...,N\}$ we have the bound 
$$(\min_j w_j) \rho_k\leq \sum_{j=1}^N w_j \rho_j =\rho\,,$$ 
and so $\frac{\rho_k}{\rho}\leq \frac{1}{\min_jw_j}$.

% Next, let $\hat{\rho}_i$ be the density of the measure $\rho_idx$ with respect to the measure $\rho dx$. In terms of densities with respect to the volume form of $\M$ this means
% \[ \hat{\rho}_i = \frac{\rho_i}{\rho}. \]
The \textit{coupling} parameter of the mixture model is defined as
\begin{equation}
 \mathcal{C}:= \max_{k=1, \dots, N} \mathcal{C}_k,
\label{Coupling-parameter}
\end{equation}
where
\begin{equation}
 \mathcal{C}_k:= 
%  \frac{1}{4} \int_{\M}  \frac{| \nabla \hat{\rho}_k |^2}{\hat{\rho_k}} \rho dx 
% = 
\frac{1}{4} \int_{\M}  \left|\frac{\nabla\rho_k}{\rho_k}-\frac{\nabla\rho}{\rho}\right|^2\rho_k dx ,
% = \frac{1}{4} \text{Fisher}_\rho(\rho_k), 
\quad k=1, \dots, N.
 \label{Coupling}
\end{equation}
This quantity is the Fisher information of the measure $\rho_kdx$ with respect to the measure $\rho dx$. We use the convention that if the density $\rho_k(x)$ vanishes at some point $x \in \M$, then we interpret the integrand in \eqref{Coupling} as zero at that point; in other words we only integrate over the set $\{ x \in \M \: : \: \rho_k(x) >0 \}$.

Finally, the \textit{indivisibility} parameter for the model is defined as
\begin{equation}\label{Lambda-definition}
\Ind := \min_{k=1, \dots, N} \Ind_k, 
\end{equation}
where 
\begin{equation}
\Ind_k :=  \min_{u \perp \mathds{1}} \frac{\int_{\M} \lvert \nabla u \rvert^2 \rho_k dx }{\langle  u,  u\rangle_{\rho_k}}.
\label{Indivisbility}
\end{equation} 
In \eqref{Indivisbility}, $\mathds{1}$ stands for the function that is identically equal to one, and $\perp$ indicates orthogonality with respect to the inner product $\langle \cdot , \cdot \rangle_{\rho_k}$. This may as well be defined as the first non-trivial eigenvalue of the operator $\Delta_{\rho_k}$.
%\red; see section \ref{sec:differentweights} for a discussion regarding the existence of a spectral gap for the operator $\Delta_{\rho_k}$.\nc

\begin{definition}[Well-separated mixture model]\label{def:well-separated-mixture}
A mixture model 
\[ \rho =  \sum_{k=1}^N w_k \rho_k \]
is well-separated if 
\[ \mathcal{S} \ll 1 , \quad \frac{\C}{\Ind} \ll 1.\]
These conditions are to be read as: the overlapping $\S$ is small enough, and the coupling parameter $\mathcal{C}$ is small in comparison to the 
indivisibility parameter $\Ind$. 
\end{definition}

In Section \ref{sec:discussionparameters} we discuss different interpretations of $\S, \C, \Ind$ and present some examples whose purpose is to illustrate the significance of all three parameters.  Notice that the notion of a well-separated model has been introduced qualitatively, matching the informal version of our first main result (Theorem \ref{thm:well-separated-gives-cone-structure-continuum}). In Theorem \ref{mainTheorem} we give a precise quantitative relation of how tightly the measure $F_{\sharp} \nu$ concentrates around orthogonal vectors depending on the quantities $\S$ and $\C/\Theta$. In other words, Theorem \ref{mainTheorem} is the quantitative (and precise) version of Theorem \ref{thm:well-separated-gives-cone-structure-continuum}.

\begin{remark}
We will use the functions $q_1, \dots, q_N$ frequently in the remainder so it is worth mentioning a simple interpretation for them. Suppose that we wanted to obtain one sample from the mixture model introduced in \eqref{DefRho}: $\x \sim \rho$. One way to obtain $\x$ is to first choose a random number $\bm{k} \in \{1, \dots, N\}$ where $\Prob(\bm{k}= k) = w_k$ and then sample $\x$ from $\rho_{\bm{k}}$. Conversely given the value $\x =x$, we can check that:
\[ \Prob(\bm{k}= k | \x=x) = \frac{w_k \rho_k(x)}{\rho(x)}.  \]
In that order of ideas, $q_k(x)$ is simply the square root of the likelihood of $\x$ being sampled from $\rho_k$. 
\end{remark}

\subsection{Geometric structure of Laplacian embeddings}
\subsubsection{The continuum setting}
\label{sec:continuum}

In what follows we describe the geometry of the ground-truth spectral embedding (i.e. the measure clustering embedding in the continuum). First, let $u_1, \dots, u_N$ be the first
$N$ orthonormal (with respect to $\langle \cdot  , \cdot \rangle_{\rho}$) eigenfunctions of $\Delta_\rho$ corresponding to its $N$ smallest eigenvalues. Define the Laplacian embedding 
\begin{equation}
  F: x \in \M \longmapsto \left( \begin{matrix} u_{1}(x) \\ \vdots \\ u_{N}(x)  \end{matrix} \right)
  \in \R^N. 
\label{ContEmbedding}
\end{equation}

The push-forward of the measure $\nu$ by $F$, denoted $\mu:= F_{\sharp } \nu $, is a measure on $\R^N$ which describes the distribution of points originally in $\M$ after transformation
by the Laplacian map $F$. Our first main result asserts that if the mixture model \eqref{DefRho} is well-separated, then the measure $\mu$ has an orthogonal cone structure; this is a geometric notion originally introduced for point clouds in \cite{BinYu} and  extended here for arbitrary probability distributions.

\begin{definition}\label{def:contcones}
Let $\sigma \in (0,  \pi /4 \nc)$, $ \delta \in [0,1)$, and $r >0$. We say that the probability measure $\mu \in \mathcal{P}(\R^k)$ has an orthogonal cone structure with parameters $(\sigma, \delta, r)$ if there exists an orthonormal basis for $\R^k$, $e_1, \dots, e_k$, such that 
\[  \mu \left( \bigcup_{j=1}^{k} C(e_j, \sigma,r)     \right) \geq 1 - \delta, \]
where $C(e_j, \sigma,r)$ is the set
\[  C(e_j , \sigma,r) := \left\{  z \in \R^k \:  : \:      \frac{ z \cdot e_j}{|z|} > \cos(\sigma), \quad  |z| > r       \right\}. \]
\end{definition}
In simple words, a measure $\mu \in \mathcal{P}(\R^k)$ has an orthogonal cone structure if it is concentrated on cones centered at vectors which form an orthonormal basis for $\R^k$.  The parameter $\sigma$ is the opening angle of the cones, $r$ is the radius of the ball centered at the origin that we remove from the cones to form the sets $C(e_j,\sigma,r)$, and $\delta$ is the amount of mass  that is allowed to lie outside the cones. The condition $\sigma<\pi/4$ ensures that the sets $C(e_j,\sigma,r)$ do not overlap.

%----------------------------------------------------------------
We are now ready to state our first main result.
\begin{theorem} \label{mainTheorem} Let $\nu$ be a mixture model with density $\rho$ of the form \eqref{DefRho} which we assume satisfies Assumptions \ref{Assump1}. For $\sigma\in(0,\pi/4)$, suppose 
\[\S <\frac {w_{min}(1-\cos^2(\sigma))}{w_{max} \cos^2(\sigma) N^2  }\,.\]
Define
\[\tau:=4  \left( 
\sqrt{\frac{\Ind(1 -N\S)}{\C}  } - \frac{\sqrt{ N \S}}{(1- \S)} \right)^{-1} +  \sqrt{\S}\,. \]
and 
\[ \delta^*:= \frac{w_{max} \cos^2(\sigma) N^2 \S }{w_{min}(1-\cos^2(\sigma))}\,.\]
Here, $w_{max}:=\max_{i  =1, \dots, N} w_i$ and $w_{min}:=\min_{i  =1, \dots, N} w_i$.  

Suppose that 
\begin{equation}\label{Ndelta<1}
 \tau-\sqrt{S}>0, \quad \tau N<1,
\end{equation}
and 
 $s, t >0$  satisfy  
\begin{equation}\label{muconehyp}
\frac{t^2\sin^2(s)}{N w_{max}} \geq N\left(\frac{\tau-\sqrt{\S}}{2}\right)^2
+ 4N^{3/2} \left(\frac{1}{\sqrt{1-N \tau}} -1 \right)\,,\qquad
s+\sigma <\frac{\pi}{4}.
\end{equation}
Then, the probability measure $\mu:=F_\sharp\nu$ has an orthogonal cone structure with parameters $\left(\sigma+s,\delta+t^2, \frac{1-\sin(s)}{\sqrt{w_{max}}}\right)$ for any $\delta\in[\delta^*,1)$.  Here, $F$ is the continuum Laplacian embedding defined in \eqref{ContEmbedding}.
\end{theorem}

%-------------------------------------------------------------------

Notice firstly that the smaller the quantities $\S$ and $\C/ \Ind$ are (i.e. the better separated the mixture model is), the smaller $\tau$ is. As a consequence the right-hand side of \eqref{muconehyp} is also small. This allows a choice of parameters $s,t$ close to zero. Secondly, for small $\S$,
we can choose $\sigma$ to be small, and if $\S$ is small compared to $\frac {w_{min}(1-\cos^2(\sigma))}{w_{max} \cos^2(\sigma) N^2  }$, then $\delta^*$ will be small also.
Following Definition~\ref{def:contcones}, small choices of $\delta$ and $t$ mean that a bigger proportion of the mass of the measure $F_{\sharp} \nu$ can be found inside the cones, whereas if $\sigma$ and $s$ are small, then the cones are more concentrated around orthogonal vectors. The bottom line is that the smaller $\S$ and $\C/ \Ind$ are, the more concentrated around orthogonal vectors the measure $F_{\sharp} \nu$ is. Note also that keeping all other parameters fixed, smaller choices of $\sigma$ will force a bigger choice of $\delta$. This reflects a trade-off between the total mass of the measure $F_{\sharp} \nu$ inside the cones, and the opening angle of the cones.

To see that the measure $F_{\sharp} \nu$ has an orthogonal cone structure under the assumption of a well-separated mixture model, we first study the geometry of a related measure $F^Q_{\sharp} \nu$ where $F^Q$ is the map
\begin{equation}
  F^Q: x \in \M \longmapsto  \left( \begin{matrix} \frac{q_1(x)}{\sqrt{w_1}} \\ \vdots \\ \frac{q_N(x)}{\sqrt{w_N}} \end{matrix} \right) \in \R^N,
  \label{FQ}
  \end{equation}
   where we recall the functions $q_k$ are as in \eqref{qk-definition}. Note that the functions $q_k$ are normalized in the sense that $\sum_{k=1}^N |q_k(x)|^2 = 1$ for all $x\in\M$, whereas the entries of $F^Q$ are normalized in $L^2(\rho)$ since
   \[\|F^Q_k\|_\rho^2=\int_\M \frac{\rho_k(x)}{\rho(x)}\,\rho(x)dx= 1 \qquad \forall \, k\in\{1,...,N\}\,.\]
   We will show that $F^Q_{\sharp} \nu $ has an orthogonal cone structure provided the overlapping parameter of the mixture model $\mathcal{S}$ is small enough (see Proposition \ref{prop:muQcone}).  In Section \ref{sec:Conic} we show that orthogonal cone structures are stable under Wasserstein perturbations. Theorem \ref{mainTheorem} is then proved by showing that the measures $F^Q_{\sharp} \nu$ and $F_{\sharp} \nu$ are  close to each other in the Wasserstein sense modulo an orthogonal transformation. To achieve this we will essentially show that $q_1, \dots, q_N$ are close to spanning the subspace generated by the first $N$
   eigenfunctions of $\Delta_\rho$.  This statement is trivially true in the extreme case when $\M$ has $N$ connected components and the $\rho_k dx$ are the uniform distributions of the connected components. Indeed, in that case, the first $N$ eigenvalues of $\Delta_\rho$ are zero, and the first $N$ eigenfunctions of $\Delta_\rho$ are rescaled versions of the indicator functions of the components. In other words, the span of the $\rho_k$ is exactly equal to the span of the first $N$ eigenfunctions of $\Delta_\rho$, and in particular up to a rotation we have $F^Q=F  $.

%%%%%%%%%%%%%%%%%%%%%%%%%%%%%%%%%%%%%%%%%%%%%%%%%%
\subsubsection{The discrete setting}\label{sec:fromconttodiscr}

Having discussed spectral clustering at the continuum level, we now move to the discrete setting. Let $\M_n := \left\{ \x_1, \dots, \x_n \right\}$ be  i.i.d. samples 
from the probability distribution $d \nu = \rho(x) dx$. We denote by $\nu_n$ the empirical measure associated to the samples $\x_1,\dots, \x_n$,
\begin{equation}
\nu_n := \frac{1}{n}\sum_{i=1}^{n}\delta_{\x_i}.
\label{Empirical} 
\end{equation}
%In particular, there exists positive numbers $0 < \theta, \Theta$ such that for every $x \in \M$,
%\begin{equation}
%\theta \leq \rho(x) \leq \Theta. 
%\label{BoundsRho}
%\end{equation}
We denote by $L^2(\nu_n)$ the space of functions $u_n : \M_n \rightarrow \R$ and identify $u_n \in L^2(\nu_n)$ with a column vector $(u_n(\x_1) , \dots, u_n(\x_n))^T $ in $\R^n$.

We associate a weighted graph structure to $\M_n$ as follows. Let $\eta: [0,\infty) \rightarrow [0,\infty)$ be a non-increasing Lipschitz function with compact support and let
\begin{equation}
\alpha_\eta:= \int_{0}^{\infty}\eta(r)r^{m+1}dr, 
\label{alphaeta}
\end{equation}
where we recall $m\leq d$ is the intrinsic dimension of the manifold $\M$ (typically we think of $m \ll d$, but this is unimportant for the purposes of the results discussed here). We assume that the kernel has been normalized so that
\[ \int_{\R^m}\eta(|x|)dx =1.  \]
For a given $\veps>0$ we define the weight matrix $\W_n$ with entries $(\W_n)_{ij}$,
\begin{equation}
(\W_n)_{ij} := \frac{2 \eta_{\veps}(\lvert \x_i-\x_j \rvert)}{ \alpha_\eta\veps^2 (d_\veps(\x_i))^{1/2} \cdot (d_\veps(\x_j))^{1/2}}, 
\label{KernelizedWeights}
\end{equation}
measuring how similar points $\x_i$ and  $\x_j$ are. Here, $|\cdot|$ denotes the Euclidean distance
in the ambient space $\R^d$, $\eta_\veps$ is defined as 
\[\eta_\veps(r):= \frac{1}{\veps^m}\eta\left(\frac{r}{ \veps}\right),\] 
and  
\[ d_\veps(y):= \sum_{j=1}^{n} \eta_\veps(|y- \x_j|)\qquad y\in\M\,.\]
 We notice that the weights in \eqref{KernelizedWeights} have been appropriately rescaled so that the eigenvalues of the graph Laplacian $\Delta_n$ defined below approximate those of $\Delta_\rho$. We would also like to notice that $(\W_{n})_{ij}$ can be written as
\[ (\W_n)_{ij}=\frac{2 \eta_{\veps}(\lvert \x_i-\x_j \rvert)}{  \alpha_\eta n \veps^2 (d_\veps(\x_i)/n)^{1/2} \cdot (d_\veps(\x_j)/n)^{1/2}}     ,\]
and that $d_\veps(\cdot) /n$ is nothing but a kernel density estimator  for the density $\rho$ due to the fact that $\eta$ has been assumed to be normalized.

The graph Laplacian operator $\Delta_n: L^2(\nu_n)\rightarrow L^2(\nu_n)$ is then defined as the matrix
\[ \Delta_n := \D_n- \W_n, \] 
where $\D_n$ is the degree matrix associated to the weight matrix $\W_n$, that is, $\D_n$ is a diagonal matrix whose diagonal entries $(\D_n)_{ii}$ are given by $(\D_n)_{ii}= \sum_{j=1}^{n}(\W_n)_{ij}$. 
Notice that here we have suppressed the dependence of $\Delta_n$ on $\veps$ for notational convenience. 

The spectrum of $\Delta_n$ induces a graph Laplacian embedding   
\[F_n : \M_n \rightarrow \R^N\]
\begin{equation}
F_n: \x_i \longmapsto \left( \begin{matrix} u_{n,1}(\x_i) \\ \vdots \\ u_{n,N}(\x_i)  \end{matrix} \right),
\label{DiscreteEmbedding}
\end{equation}
where $u_{n,1}, \dots, u_{n,N}$ are the first $N$ eigenvectors of $\Delta_n$.

\begin{remark}
The normalization terms appearing in \eqref{KernelizedWeights} induce the \textit{kernelized} graph Laplacian $\Delta_n$. There are other possible normalizations that one can consider such as the ones in the parametric family introduced in \cite{Coifman1}. The effect of the different normalizations can be observed asymptotically (as $n \rightarrow \infty$) in the way the density $\rho$ affects the differential operator $\Delta_\rho$  (which in general will differ from \eqref{ContLaplaciansRho}). We expand our discussion on the choice of different normalizations in Section \ref{sec:differentweights}; for an overview of a more general family of normalizations see \cite{FHBHAS}. In the context of this paper, it is enough to say that $\Delta_n$ and $\Delta_\rho$ (as defined in \eqref{ContLaplaciansRho}) are closely related to each other according to Theorem \ref{TheoremTranps}.
\end{remark}

In order to state and prove our second main result we make some additional assumptions on the ground-truth measure $\nu$, on the number of data points $n$ and the connectivity parameter $\veps$.

\begin{assumption}\label{assumption-on-rho-contrast}
We assume that $\M$ is a smooth, compact, orientable, connected $m$-dimensional manifold without a boundary. Furthermore, we assume that $\rho\in C^1(\M)$, and that there exist positive numbers $0<\rho^-< \rho^+$ such that 
$$
\rho^- \le \rho(x) \le \rho^+, \qquad \forall x \in \mathcal{M}.
$$
We denote by $C_{Lip}>0$ the Lipschitz constant of $\rho$.

We also assume that $n$ is large enough and $\veps$ is small enough, so that 
\begin{equation}
\label{ineq:assump}
\veps\left(1+ \sqrt{\lambda_N}\right) + \frac{\log(n)^{p_m}}{n^{1/m}\veps} \leq \min \{ C, \frac{1}{2}(\lambda_{N+1}- \lambda_N) \} \,,
\end{equation}
where $p_m= 3/4$ if $m=2$ and $p_m= 1/m$ when $m\geq 3$,  $C>0$ is
a finite  constant that depends only on $\M$, and $\lambda_N$ and $\lambda_{N+1}$ are the $N$-th and $(N+1)$-th eigenvalues of $\Delta_\rho$
in increasing order.
\end{assumption}

\begin{remark}
Several theoretical results deducing convergence rates for the spectra of graph Laplacians towards their continuum counterparts make similar assumptions to those in Assumption \ref{assumption-on-rho-contrast} (e.g \cite{Moritz,GTSSpectralClustering}), and for the moment being we will only be able to state our second main result under the assumption that $\M$ is compact (see Section \ref{UnboundedDomains} where we discuss further extensions and generalizations).
Regarding the assumptions on $n$ and $\veps$, we remark that requiring the quantity on the left hand side to be less than a constant $C$ is the same  condition used in the literature to quantify rates of convergence for the spectra of graph Laplacians (see for example \cite{BIK,Moritz}). On the other hand, requiring the left hand side of \eqref{ineq:assump} to also be smaller than $\frac{1}{2}(\lambda_{N+1}-\lambda_{N})$ is a condition that is easier to satisfy when the model is well-separated. Indeed, we will later see that it is possible to find an upper bound for $\lambda_N$ (Lemma \ref{UppLambdaN}) and a lower bound for $\lambda_{N+1}$ (Proposition \ref{auxLemmaSigmaB}) in terms of the parameters of the mixture model. That is, one can restate our assumption and write an inequality in terms of the overlapping, indivisibility and coupling parameters.  However, at this point we do not believe there is any benefit in developing this discussion any further, and we remark that these are simply concrete and technical conditions that guarantee that we have entered the regime where our theorems hold.
\end{remark}

\begin{remark}
Notice that when the manifold $\M$ is smooth and compact as in Assumptions \ref{assumption-on-rho-contrast}, Assumption \ref{Assump1} is automatically satisfied.   
\end{remark}

\begin{remark}
Notice that the fact that $\rho$ is bounded away from zero does not imply that the individual components $\rho_k$ are as well. In particular,  an individual component $\rho_k$ may be arbitrarily close to zero (or even equal to zero) in certain region of $\M$ without contradicting Assumption \ref{assumption-on-rho-contrast}. As a matter of fact, one precisely needs this to be the case in order for a model to be well-separated (for otherwise $\S$ would not typically be small).
\end{remark}

In Section \ref{App:A}, we make the connection between the spectrum of the graph Laplacian $\Delta_n$ and the spectrum of the continuous operator $\Delta_\rho$. The proof of the main technical result, Theorem~\ref{TheoremTranps}, relies on a comparison between the Dirichlet forms associated to the operators $\Delta_n$ and $ \Delta_\rho$ by way of careful interpolation and discretization of discrete and continuum functions. 
Only small modifications to the proofs in \cite{BIK,Moritz} are necessary and Section~\ref{App:A} should be understood as a short guide on how to adapt the machinery developed there to the context of this paper.

Finally, we leverage Theorems \ref{mainTheorem} and \ref{TheoremTranps} to establish our second main result, which asserts that under Assumption \ref{assumption-on-rho-contrast} and provided the mixture model is well-separated, the measure $F_{n \sharp}\nu_n $ also has an orthogonal cone structure with very high probability. More precisely we have the following.

%-----------------------------------------------------------------------------------
\begin{theorem}
\label{PropMuMun}
Let $N\ge 2$ and $\beta>1$.
Suppose that $\M$ and $\rho$ satisfy Assumption \ref{assumption-on-rho-contrast} for some $\veps>0$. Assume also that the parameters $\S, \C , \Ind$ of the mixture model $\rho$ are such that the quantities $\tau, \delta^*$ and $\sigma$ from Theorem \ref{mainTheorem} satisfy \eqref{Ndelta<1} and \eqref{muconehyp}. In addition, assume $\S<N^{-2}$.
Let $\x_1, \dots, \x_n$ be i.i.d. samples from the measure $d\nu = \rho dx$ with associated empirical measure $\nu_n$,  and let $F_n$ be the  Laplacian embedding defined in \eqref{DiscreteEmbedding}. 

Then, there exists a constant $C_\beta>0$ depending only on $\beta$, so that with probability at least $1- C_\beta n^{-\beta}$, the probability measure $\mu_n:=F_{n\sharp}\nu_n$ has an orthogonal cone structure with parameters 
\[\left(\sigma+s,\delta+t^2, \frac{1-\sin(s)}{\sqrt{w_{max}}}\right)\],
for any $\delta\in[\delta^*,1)$, and $s, t>0$ satisfying
\[ \sigma + s <\frac{\pi}{4}, \]
\begin{align}
\begin{split}
\label{ineq:mainTheorem2}
\frac{t\sin(s)}{\sqrt{N w_{max}}} \geq & 
\sqrt{N\left(\frac{\tau-\sqrt{\S}}{2}\right)^2
+ 4N^{3/2} \left(\frac{1}{\sqrt{1-N \tau}} -1 \right)  } 
\\ &+  \sqrt{N \phi(\S, \C, \Theta, N, \veps, n, m)},
\end{split}
\end{align}
where 
%\begin{align}
%  \phi(\S, \C, \Theta, N, \veps, n, m)& := C_\M
%  \Bigg[ \left(\frac{N\C}{1- N \S^{1/2}} \right)
%  \left( \veps +  \frac{\log(n)^{p_m}}{\veps n^{1/m}}
%    + \frac{\log(n)^{p_m}}{ n^{1/m}} \left(\frac{N\C}{1- N \S^{1/2}} \right)^{1/2} \right)^2 \\
%       & + N \left( \left( 1 - N
%             \left(\frac{N\C}{1- N \S^{1/2}} \right)^{1/2} \left( \veps +
%           \frac{\log(n)^{p_m}}{\veps n^{1/m}} + \frac{\log(n)^{p_m}}{ n^{1/m}}\right)
%         \right)^{-1} - 1 \right) \Bigg],
%\end{align}
\begin{align*}
 \phi(\S, \C, \Theta, N, \veps, n, m)& = c_\M
\left(\frac{N\C}{1- N \S^{1/2}} \right)
\left( \veps +  \frac{\log(n)^{p_m}}{\veps n^{1/m}}\right) \\
& \cdot  \left( \left( \sqrt{\Theta (1 - N \S)} - \frac{\sqrt{\C N \S}}{1- \S}\right)^2 - \frac{N \C}{1- N \S^{1/2}} \right)^{-1} ,
\end{align*}  
%\begin{align}
% C_\M &= C_\M(\Ind, \S, \C, N, \rho, \eta)\notag\\
% &= c_\M \left(\frac{(1-N\S^{1/2})}{N} \left(\sqrt{\frac{\Ind(1-N\S)}{\C}}-\frac{\sqrt{N\S}}{(1-\S)}\right)^2-1\right)^{-1}
%\label{def:CM}
%\end{align} 
and $c_\M>0$ is a constant depending on $\M$, $\rho^\pm$, $C_{Lip}$, $\eta$
and $N$. In the above we require $\S$ and $\C/\Ind$ small enough so that all the inner terms in the definition of $\phi$ make sense and are positive.
\end{theorem}
%---------------------------------------------------------

%As one would intuitively expect for a well-separated mixture model with $\S\ll1$ and $\C/\Ind\ll 1$, the first term on the right-hand side of \eqref{ineq:mainTheorem2} as well as the constant $C_\M$ are small, and so we can take the parameters $s,t$ to be small as well.

Notice that the only extra term in inequality \eqref{ineq:mainTheorem2} that does not appear in \eqref{muconehyp} is the term
$\phi(\S, \C, \Theta, N, \veps, n, m)$. 
As we will see in Section \ref{sec:proof2} (following the estimates in Section \ref{App:A}) this term is an upper bound for the Wasserstein distance between the measures $F_{\sharp}\nu$ and $F_{n \sharp} \nu_n$ (up to an orthogonal transformation). The assumptions in Theorem \ref{PropMuMun} imply that $F_{\sharp} \nu$ has an orthogonal cone structure as described in Theorem \ref{mainTheorem}, and hence in order to deduce that $F_{n \sharp} \nu_n$ has an orthogonal cone structure (as it is stated in Theorem \ref{PropMuMun})  it is enough to combine Proposition \ref{LemmaConeWasserstein} with the estimates on the Wasserstein distance between $F_{n\sharp} \nu_n$ and $F_{\sharp} \nu$.

We observe that asymptotically, if the parameter $\veps>0$ scales with $n$ like
\[ \frac{\log(n)^{p_m}}{n^{1/m} } \ll \veps \ll 1,  \] 
then $\phi(\S, \C, \Theta, N, \veps, n, m)$ gets smaller as $n \rightarrow \infty$, and we deduce that the geometry of the point cloud $F_n(\M_n)$ converges to that of $F_{\sharp} \nu$. Nevertheless, it is worth emphasizing that our results are meaningful in the finite (but possibly large) $n$ case. We also highlight that the term
$\phi(\S, \C, \Theta, N, \veps, n, m)$ suggests taking $\veps := \left(\log(n)^{p_m}/n^{1/m}\right)^{1/2}$ as an optimal choice for $\veps$. However, we must be careful with this conclusion as it relies on the estimates presented in Theorem \ref{TheoremTranps} which are not necessarily sharp (but are to the best of our knowledge the best available). This also means that any future improvements to the estimates in \cite{BIK,Moritz} will directly translate into an improvement of the estimates in Theorem~\ref{PropMuMun}.

\section{Discussion} 
\label{sec:discussion}

The purpose of this section is to provide a more detailed account of the notions introduced above, as well as to discuss our main results and provide several remarks on extensions and generalizations.

\subsection{Overlapping, coupling and indivisibility parameters}
\label{sec:discussionparameters}

We now give a more detailed description of the parameters $\S$, $\C$ and $\Ind$ introduced in Section~\ref{sec:wellsepa}. 

The overlapping parameter $\S$ is small if the components $\rho_k$ have well-separated mass in the sense that if at some point $x\in\M$ one component $\rho_k(x)$ is large, then all other components $\rho_i(x)$, $i\neq k$, are small at that point. In this case, at $x\in\M$, we have $\rho(x)\approx w_k\rho_k(x)$, and so 
$$
\frac{\rho_k(x)\rho_i(x)}{\rho(x)} \approx \frac{\rho_i(x)}{w_k},
$$
which is small.
If $\rho_k$ is small at $x\in\M$ on the other hand, then 
$$
\frac{\rho_k(x)\rho_i(x)}{\rho(x)} \leq \frac{\rho_k(x)}{w_i}
$$
is also small.
In other words, the parameter $\S$ being small requires the components not to \emph{overlap} too much (hence the name given to it).\\

If the indivisibility parameter $\Ind$ is small, this is an indication that our choice of mixture model $\rho(x)=\sum_k w_k\rho_k(x)$ for a given density $\rho$ is not optimal in the sense that it could be split into smaller meaningful components. This property is captured by the spectrum of the operators $\Delta_{\rho_k}$. For any $k\in\{1,...,N\}$, and $u,v\in C^1(\M)$, 
$$
\langle  \Delta_{\rho_k}u, v \rangle_{\rho_k}
= \int_\M \nabla u\cdot \nabla v \, \rho_k dx
=\langle \nabla u, \nabla v \rangle_{\rho_k},
$$
and so the first eigenvalue is zero with constant eigenfunction $\mathds{1}$. For the sake of illustration, suppose that the support of $\rho_k$ has two disconnected components. In that case, one can construct an eigenfunction orthogonal to $\mathds{1}$ that takes different constant values on the two different components, and so the second eigenvalue is zero. 
%This cannot happen under our assumption that $\rho_k>0$ on $\M$ and so the second eigenvalue of $\Delta_{\rho_k}$ remains strictly positive. 
In light of the max-min theorem (see for example \cite[Thm.~8.4.2]{Buttazzo}, also known as Courant--Fisher theorem) and  definition \eqref{Indivisbility}, we 
observe that $\Ind_k$ is exactly the second eigenvalue of the operator $\Delta_{\rho_k}$.

It is well-known, in both the spectral geometry and
machine learning literature, that the first
non-trivial eigenvalue of the Laplacian associated to a
geometric object (discrete or continuum) is intimately related
to the problem of two-way clustering. Having a small first non-trivial eigenvalue
indicates the presence of two clusters in the dataset in the discrete setting. 
 This is essentially because the eigenvalue 
problem can be interpreted as a relaxation of balanced cut minimization problems (see \cite{vonLux_tutorial} and also Cheeger's inequality for graphs  \cite[Sec.~2.3]{chung1997spectral}). 
The more $\rho_k$ concentrates on two separate components, the smaller the second eigenvalue will become (see \cite{FHBHAS} for precise estimates). This is why the second eigenfunction can be interpreted as an approximate identifier for clusters in the problem of two-way clustering, and it is usually referred to as the \emph{Fiedler vector}. 
Intuitively, a low value of $\Ind_k$ indicates that the component $\rho_k$ can be split into further significant components. For this reason, we require that $\Ind$, the {indivisibility} parameter for the model to be large enough, guaranteeing that the components in the model cannot be divided into further meaningful components. 
This follows from the observation that if there exists a $k\in\{1,...,N\}$ such that
$\rho_k$ can be split into separate  components, then the second eigenvalue of
$\Delta_{\rho_k}$ will be small resulting in $\Ind$ being small as well.  \\

The coupling parameter $\C$ being small can be interpreted as a metastability condition on the relative entropy of measures $\rho_k dx$ with respect to the measure $\rho dx$.  
We define the relative entropy of a probability measure $\varrho dx$ with respect to $\rho dx$ by
\begin{equation*}
 H(\varrho|\rho) 
 := \int_\M \left(\frac{\varrho}{\rho}\right)\log\left(\frac{\varrho}{\rho}\right)\,\rho dx
 = \int_\M \varrho\left(\log \varrho-\log\rho\right)\,dx\,.
\end{equation*}
Given a probability density $\rho$, consider $\varrho(t,x)$ varying with time that satisfies the evolution equation
\begin{align}\label{FP}
  \partial_t \varrho &= \Delta \varrho +\divergence(\varrho\nabla
                       \log\rho)\notag\\
 &= \divergence\left(\varrho\nabla\left(\log \varrho - \log \rho\right)\right),
\end{align}
with initial condition $\varrho(0,x)=\rho_k(x)$ for a fixed $k\in\{1,...,N\}$. 
Model~\ref{FP} is a linear Fokker-Planck equation that is driven by the competition between linear diffusion and a confining drift term given by the potential $\log\rho$. Observe that $\rho$ is a stationary state of the system, and the quantity $H(\varrho(t)|\rho)$ gives a sense of how far $\varrho(t,x)$ is from the stationary state $\rho(x)$ at time $t\geq 0$. In this context, the coupling parameter $\C_k$ can be interpreted as an initial rate of decay for the entropy. More precisely, following Boltzmann's $H$-theorem, the \emph{relative Fisher Information} $I(\varrho|\rho)$ is defined as the entropy dissipation along solutions to \eqref{FP},
\begin{align*}
 \frac{d}{dt} H(\varrho(t)|\rho)
 &= \int_\M \partial_t \varrho \left(\log \varrho-\log\rho\right)\, dx\\
 &=- \int_\M \varrho\left|\nabla \left(\log \varrho-\log\rho\right)\right|^2\, dx 
 =: -I(\varrho(t)|\rho)\,,
\end{align*}
and so the definition of the coupling parameter \eqref{Coupling} can be interpreted as
$$
\mathcal{C}_k
=\frac{1}{4} I(\varrho(0)|\rho)
=\frac{1}{4} \int_{\M}  \left|\frac{\nabla\rho_k}{\rho_k}-\frac{\nabla\rho}{\rho}\right|^2\rho_k dx \,.
$$
Approximating the entropy for small initial times $t>0$, 
\begin{align*}
 H(\varrho(t)|\rho)
 =H(\rho_k|\rho)-4\C_k t+O(t^2)\,,
\end{align*}
we observe that if the coupling parameter $\C_k$ is small, then
$H(\varrho(t)|\rho)$ only varies very slowly in a neighborhood of $t=0$, and so we can consider $H(\rho_k|\rho)$ to be in a metastable state. For a well-separated mixture model, we expect the initial entropy $H(\rho_k|\rho)$ to be large enough as small entropy would indicate that $\rho_k$ takes values close to $\rho$ (note that $H(\varrho|\rho)=0$ if and only if $\varrho=\rho$ almost everywhere). In other words, having a small overlapping parameter $\S$ indicates that the initial entropy $H(\rho_k|\rho)$ cannot be too small. For well-separatedness, however, we require that $\C_k$ is small for all $k\in\{1,...,N\}$, that is to say that the Fisher Information is small initially when starting the evolution process \eqref{FP} at any of the components $\rho_k$. In this sense, we require the mixture model to be in a metastable state.\\

%%%%%%%%%%%%%%%%%%%%%%%%%%%%%%%%%%%%%%%%%%
Having discussed the parameters $\S, \Ind$ and $\C$, we now give some intuition on the notion of a well-separated mixture as introduced in Definition~\ref{def:well-separated-mixture}. Roughly speaking, a mixture model is well-separated if it has small coupling and overlapping 
parameters $\mathcal{C}$ and $\mathcal{S}$,  but a large indivisibility parameter $\Ind$.   
A natural question is whether it is necessary to require both $\mathcal{C}$ and $\mathcal{S}$
to be small. Intuitively, $\mathcal{S}$ is small when the components $\rho_j$ are concentrated 
on disjoint sets that are far apart. In such a setting one expects the coupling parameter 
$\mathcal{C}$ to also be small. However, we present two examples where $\mathcal{C}$ and 
$\mathcal{S}$ are not necessarily simultaneously small, demonstrating that it is not enough to only impose a small $\C$ (Example \ref{example:gaussian-mixture}) or a small $\S$ (Example \ref{example:dumbell-partitioning}). The examples were chosen to illustrate how to estimate the parameters $\S$, $\C$ and $\Ind$ in some concrete settings, and to see that our mathematical definitions do capture our intuitive interpretation of what well-separated components should look like.

%%%%%%%%%%%%
%%%%%%%%%%%%%%%%%%%%%%%
\begin{example}[Mixture of two  Gaussians]\label{example:gaussian-mixture}
  Consider a mixture of two standard Gaussian densities on
  $\mathbb{R}$ obtained by shifting the two densities. More precisely,
  let $\rho= \frac{1}{2}\rho_1 + \frac{1}{2}\rho_2$, where

        \begin{equation*}
          \rho_1(x) := \frac{1}{\sqrt{2 \pi}} e^{-x^2/2}  , \quad  \rho_2(x) := \frac{1}{\sqrt{2 \pi}} e^{-(x-\gamma)^2/2}.
        \end{equation*}
 The mixture model is illustrated in Figure \ref{Gaussians}.
\begin{figure}
\includegraphics[width=0.6\textwidth, trim= 0.5cm 0.5cm
0.5cm 0.5cm]{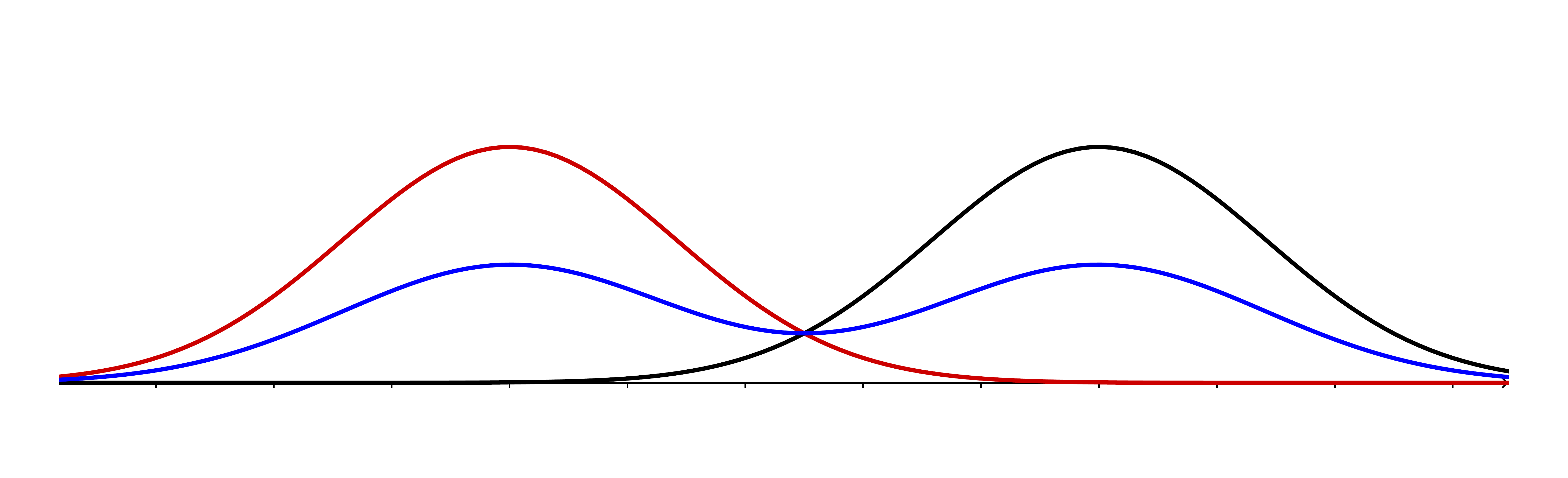}
\put(-180,35){\red $\rho_1$ \nc}
\put(-42,35){ $\rho_2$ }
\put(-150,32){\blue $\rho$ \nc}
\caption{A mixture model with Gaussian components.}
\label{Gaussians}
\end{figure}
We estimate how the overlapping and coupling parameters of the
        model scale with the off-set $\gamma$.
        Intuitively, the components become better separated as
        $\gamma$ increases and we will show that both $\C$ and $\S$
        become small in this regime. However, as $\gamma$ becomes
        small the coupling parameter $\C$ decreases once more while
        the overlapping parameter $\S$ increases. This example
        demonstrates why we require both $\C$ and $\S$ to be small.

Observe that the indivisibility parameter is unaffected by
        shifts in the $x$-axis, i.e., $\Ind$ does not depend on
        $\gamma$. On the other hand, straightforward calculations show
        \begin{equation*}
          \frac{\rho_1'(x)}{\rho_1(x)} =  -x , \quad \frac{\rho_2'(x)}{\rho_2(x)} = -(x- \gamma),
        \end{equation*}
and that
        \begin{equation*}
  \frac{\rho'(x)}{\rho(x)} = -x + \frac{\gamma}{2} \frac{\rho_2(x)}{\rho(x)} = -(x-\gamma) - \frac{\gamma}{2} \frac{\rho_1(x)}{\rho(x)}.
        \end{equation*}
In particular,
        \begin{equation*}
  \frac{\rho'(x)}{\rho(x)} - \frac{\rho_1'(x)}{\rho_1(x)} =  \frac{\gamma}{2} \frac{\rho_2(x)}{\rho(x)} , \quad  \frac{\rho'(x)}{\rho(x)} - \frac{\rho_2'(x)}{\rho_2(x)} = -\frac{\gamma}{2}\frac{\rho_1(x)}{\rho(x)}.
        \end{equation*}
        It then follows from \eqref{Coupling} that
%         \begin{equation*}
%           \mathcal{C} = \int_{\mathbb{R}} \frac{\gamma^2}{4} \frac{\rho_2^2(x)}{\rho^2(x)} \rho_1(x) dx
%           \le \frac{\gamma^2}{2}  \mathcal{S}.
%         \end{equation*}
    \begin{equation*}
          \C \leq \C_1+\C_2
          = \frac{\gamma^2}{16}\int_{\mathbb{R}}  \frac{\rho_1 \rho_2}{\rho^2} \left(\rho_1+\rho_2\right) dx
          = \frac{\gamma^2}{8}\int_{\mathbb{R}}  \frac{\rho_1 \rho_2}{\rho} dx
          = \frac{\gamma^2}{8}  \mathcal{S}.
        \end{equation*}
        Thus, the coupling parameter $\C$ is controlled by the
        overlapping parameter $\mathcal{S}$ and $\gamma^2$.  It
        follows that when $\gamma$ is close to zero, the overlapping parameter is close to one since both $\rho_1$ and $\rho_2$ are roughly equal to $\rho$, and that the coupling
        parameter is small of order $\gamma^2$. On the other hand, for large
        $\gamma$, the overlapping parameter decays exponentially
        fast and the coupling parameter decays as well. From an intuitive point of view it is to be expected that for large $\gamma$ the model is well separated. 
%         \hfill\ensuremath{\blacksquare}
      \end{example}
%%%%%%%%%%%%%%%%%%%
%%%%%%%%%%%%%%%%%

%%%%%%%%%%%%%%%%%%
%%%%%%%%%%%%%%%%%%
\begin{example}[Partitioning a dumbbell shaped domain]\label{example:dumbell-partitioning}
We consider a probability measure $\rho$ on a dumbbell shaped domain
$\mathcal{M} \subset \R^2$. More precisely define the sets
\begin{align*}
  \mathcal{M}_0 & = \{ (x_1,x_2) : -\ell/2 < x_1  <\ell/2, \qquad |x_2| < \vartheta \},  \\
  \mathcal{M}_1 & = \{ (x_1,x_2) : \ell/2 < x_1 < \ell, \qquad |x_2| < \ell/2 \},  \\
  \mathcal{M}_2 & = \{ (x_1,x_2) : -\ell < x_1 < -\ell/2, \qquad |x_2| < \ell/2 \},
\end{align*}
for $\ell,\vartheta >0$ and define
$\mathcal{M} := \bigcup_{j=0}^2 \mathcal{M}_j\subset \R^2$.  We have in mind
$\vartheta \ll \ell$ so that $\mathcal{M}$ is a dumbbell shaped domain
consisting of two rectangles connected by a thin rectangle.

We assume $|\mathcal{M}| = 1$, i.e. $\ell^2 + 2\ell \vartheta =1$. We take $\rho =1$ on
$\mathcal{M}$ and then define the mixture elements $\rho_k$ by
partitioning $\mathcal{M}$ into subsets and taking the $\rho_k$ to be
a partition of unity for $\mathcal{M}$ with appropriate
normalization. We present two choices of partitions and
demonstrate that for small $\vartheta$, one choice results in
well-separated densities $\rho_k$  while
the other partition does not.

Our first partition aims at cutting the dumbbell
perpendicularly to its axis (see
Figure~\ref{fig:dumbell-partition}(A)). For $0<\veps<l/2$ define the function
%sets
%\begin{align*}
%  {\mathcal{M}}_+  = \{ (x_1,x_2) \in \mathcal{M} : x_1 \ge 0\}, \qquad  
%  {\mathcal{M}}^\veps_+ & = \{ (x_1,x_2) \in \mathcal{M} : x_1 \ge -\varepsilon\}\,, \\
%  {\mathcal{M}}_-  = \{ (x_1,x_2) \in \mathcal{M} : x_1 \le 0\}, \qquad 
%  {\mathcal{M}}^\veps_- & = \{ (x_1,x_2) \in \mathcal{M} : x_1 \le \varepsilon\}\,, \\
%\end{align*}
%and let 
$\psi: \R \rightarrow \R $
\[\psi(t) := \begin{cases} 1 & \text{ if } t \leq 0 \\ 1- \frac{2}{\veps^2}t^2 &\text{ if } t \in [0, \veps/2] \\ \frac{2(t-\veps)^2}{\veps^2}& \text{ if }t \in [\veps/2, \veps]  \\
0 & \text{ if } t \geq \veps \end{cases}. \]
We let $\tilde{\rho}_1,\tilde{\rho}_2: \M \rightarrow \R$ ,
\[ \tilde{\rho}_1(x):= \psi(x_1),\quad  \tilde{\rho}_2(x) := \psi(-x_1),\quad x=(x_1, x_2) \in \M. \]

The densities $\rho_1$ and $\rho_2$ are then defined by
\begin{equation*}
  \rho_1(x) := \frac{2\tilde{\rho}_1(x)}{\tilde{\rho}_1(x)+ \tilde{\rho}_2(x)}, \qquad 
  \rho_2(x):=\frac{2\tilde{\rho}_2(x)}{\tilde{\rho}_1(x)+ \tilde{\rho}_2(x)},
\end{equation*}
from where we see that $\rho(x) = \frac{1}{2} \rho_1(x) + \frac{1}{2} \rho_2(x)$
with weights $w_1 = w_2 =\frac{1}{2}$ in the notation of \eqref{DefRho}.

We now show that the above mixture model is well-separated according
to Definition~\ref{def:well-separated-mixture} for an appropriate choice of
$\varepsilon(\vartheta)$ as $\vartheta \to 0$. Let us begin by estimating
the parameters $\mathcal{S}, \mathcal{C}$ and $\Ind$. Following
\eqref{eqn:S2} we have
\begin{align}\label{dumbell-S-parameter}
  \begin{split}
    \mathcal{S} & =  \int_{\mathcal{M}}\rho_1(x) \rho_2(x) dx \\
    & \leq 4
    \int_{\mathcal{M}} \tilde{\rho}_1(x) \tilde{\rho}_2(x) dx \\
    & \le 16\vartheta\veps
  \end{split}
\end{align}
where in the first inequality we have used the fact that
\[  1 \leq \tilde{\rho}_1(x) + \tilde{\rho}_2(x) \leq 2. \]
% \fhremark{I replaced $\nu$ with $\vartheta$ since we already have $d\nu=\rho dx$.}
Thus $\mathcal{S}$ can be made arbitrarily small by either taking $\vartheta$ or $\veps$ to zero.  Furthermore, by \eqref{Coupling} one can see that
\begin{align*}
  \mathcal{C} & = \max_k  \frac{1}{4} 
              \int_{\mathcal{M}} \left|\frac{\nabla \rho_k(x)}{\rho_k(x)}\right|^2 \rho_k(x) dx \leq 4 \vartheta \int_{0}^{\veps}\left \lvert \frac{\psi'(t)}{\psi(t)} \right  \rvert^2 \psi(t) dt \leq  \frac{4\vartheta}{\veps}.
\end{align*}
This implies that if we pick 
\[  \vartheta \ll \veps,\]
then we can make $\C$ as small as we want. It remains  to show that the
indivisibility parameter $\Ind$ remains bounded away from zero for the right choice of parameters. In order to find a lower bound for $\Theta$ it will be convenient to use Cheeger's inequality, which gives us a lower bound for $\Theta$ in terms of a geometric quantity that is much easier to understand and estimate. Namely, by Theorem~\ref{thm:Cheeger},
\begin{equation}
 \Theta_k = \min_{u \perp \mathds 1}  \frac{\int_{\M} |\nabla u |^2 \rho_k dx }{\int_{\M}u ^2\rho_k dx }\geq \frac{ (h(\M, \rho_k))^2}{4}, 
 \label{Ineq:Cheeger}
 \end{equation}
where in the above $\perp$ means orthogonal with respect to the inner product $\langle \cdot, \cdot\rangle_{\rho_k}$ and 
\[ h(\M, \rho_k):= \min_{A \subseteq \M} \frac{\int_{\partial A \cap \M} \rho_k(x)dS(x) }{ \min \left\{ \int_{A}\rho_k dx , \int_{\M \setminus A} \rho_k dx \right\}} . \]
The inequality is due to Cheeger \cite{Cheeger}, and for the convenience of the reader  we provide in the Appendix a sketch of the proof that in particular highlights the specific structure of the Raleigh quotient associated to the operators $\Delta_{\rho_k}$ and $\Delta_\rho$ that makes the inequality work (see Section~\ref{sec:Cheeger}). In particular we emphasize that for Cheeger's inequality to hold, both the numerator and denominator of the Raleigh quotient defining $\Theta_k$ must be weighted with the same measure (in our case both integrals are weighted by $\rho_k$). 

Due to symmetry we have $\Theta=\Theta_1=\Theta_2$. Thanks to Cheeger's inequality we can focus on obtaining a lower bound for $h(\M , \rho_1)$. Now, from the fact that $ 2 \geq \tilde{\rho}_1(x) + \tilde{\rho}_2(x) \geq 1 $ we see that in order to get a lower bound for $h(\M, \rho_1)$ we can alternatively get a lower bound for $h(\M, \tilde{\rho}_1)$. After a moment of reflection, we can conclude that if we can find $C(\vartheta, \veps)$ such that 
\[  Cut(A_t):=  \frac{\int_{\partial A_t \cap \M} \tilde \rho_1(x)dS(x) }{ \min \left\{ \int_{A_t} \tilde \rho_1 dx , \int_{\M \setminus A_t} \tilde \rho_1 dx \right\}}  \geq C(\vartheta, \veps) , \quad \forall t \in [ -l/2 , \veps],  \]
where
\[ A_t := \{ x \in \M \: : \:  x_1 \leq t \} ,\]
then 
\[ h(\M, \rho_1) \geq \min \{ c, C(\vartheta,\veps)  \} ,\]
where $c$ is a constant that does not depend on $\vartheta$ nor $\veps$ and
corresponds to the cost of a cut along the $x_1$ axis. This is because other than the
horizontal cut along the $x_1$ axis,
the only other natural cuts that are viable competitors for minimizing $Cut$ 
(the objective function in the definition of $h(\M,\tilde \rho_1)$) are the ones perpendicular to
the bottleneck of the dumbbell and along the $x_2$ axis. A straight forward calculation shows that the cuts along the $x_1$ axis result in a value of $Cut$ that is bounded away from zero  and so  we focus on finding $C(\vartheta,\veps)$ by analyzing vertical cuts across the bottleneck.
Notice that we do not have to consider the case $t > \veps$ since the density $\tilde{\rho}_1$ is equal to zero for $x=(x_1, x_2)$ with $x_1 >\veps$. In other words we just have to look at non-trivial subsets of $\{ x \in \M \: : \: \tilde{\rho}_1(x) >0 \}$.

For all $t \in [-l/2, \veps]$ we have $\partial A_t \cap \M=\{(x_1,x_2)\in\M\,:\, x_1=t, \,|x_2|<\vartheta\}$, and so we can write
\[Cut(A_t) = \frac{\psi(t)}{ \int_{t}^\veps \psi(s) ds}. \]
Now, when $\veps$ is small, for every $t \in [-l/2, \veps/2]$ we have
\[ Cut(A_t) \geq \frac{1}{2 \int_{t}^{\veps} \psi(s) ds  } \geq 1. \]
On the other hand,  we notice that by the definition of $\psi$ we have
\[ \psi'(t) \int_{t}^\veps \psi(s) ds + (\psi(t))^2 
= \frac{4(\varepsilon-t)^4}{3\varepsilon^4}\geq 0, \quad \forall t \in [\veps/2, \veps], \]
implying that the function $Cut(A_t)$ is increasing between $\veps/2$ and $\veps$. In particular, it follows that for all $t \in [\veps/2, \veps]$ we have
\[  Cut(A_t) \geq Cut(A_{\veps/2}) = \frac{1}{2 \int_{\veps/2}^\veps \psi(s) ds } \geq \frac{1}{\veps}.\]
Hence $C(\vartheta,\varepsilon)=\min\{1,1/\varepsilon\}$.
We conclude that provided  $\veps\leq 1$ we have that
\[ \Theta \geq c, \]
for some positive constant that does not depend on $\veps$ or $\vartheta$. The bottom line is that the model is well-separated as long as %$\vartheta \ll \veps < 1 \wedge \ell/2$ and $\veps$ is
$\vartheta \ll \veps$ and $\veps$ is
sufficiently small.

One can then consider a partition that aims at cutting the dumbbell along its axis (see Figure~\ref{fig:dumbell-partition}(B)). In particular we can introduce analogue densities $\rho_1$ and $\rho_2$ smoothening out the indicator functions for the sets in the new partitioning (in this case the smoothening happens in the $x_2$ coordinate). Then using similar calculations as before we can
show that the overlapping parameter can be made arbitrarily small (by selecting the parameter $\veps$ to be small), however, it will not be possible to control the indivisibility parameter from below as $\vartheta \rightarrow 0$. In fact,  we expect the second eigenvalue of $\Delta_{\rho_k}$ to go to
zero with $\vartheta$ following~\cite{anne1995note}. Moreover, the coupling parameter will not be controlled by $\vartheta$ in this case, and in fact as $\veps \rightarrow 0$ it will blow up. Therefore according to our definition this mixture model is not well-separated. 

Notice that this example has made explicit the fact that the notion of well separated mixture model does not depend on the density $\rho$, but on the actual decomposition as a mixture model. The example has also exhibited that our mathematical notion of well-separated mixture model does capture our intuition.
\end{example}

%%%%%%%%%%%%%%%%%%
%%%%%%%%%%%%%%%%

        \begin{figure}[htp]
          \centering
          \begin{subfigure}[b]{.49 \textwidth}
          \includegraphics[height=.8\textwidth, trim
          =2cm 4cm 2cm 4cm, clip]{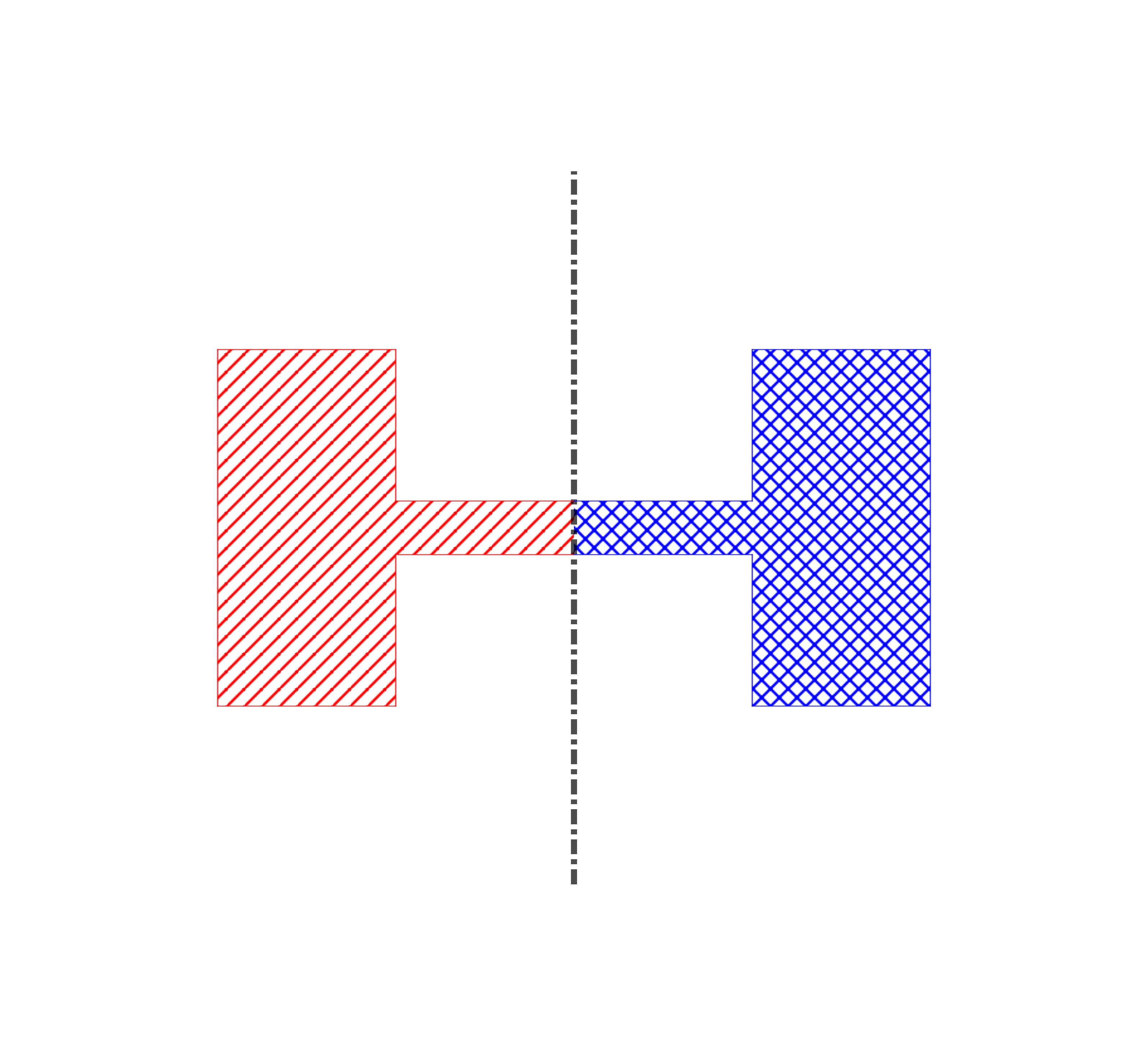}
          \caption{}
          \end{subfigure}
          \begin{subfigure}[b]{.49\textwidth}
          \includegraphics[height=.8 \textwidth, trim
          =2cm 4cm 2cm 4cm, clip]{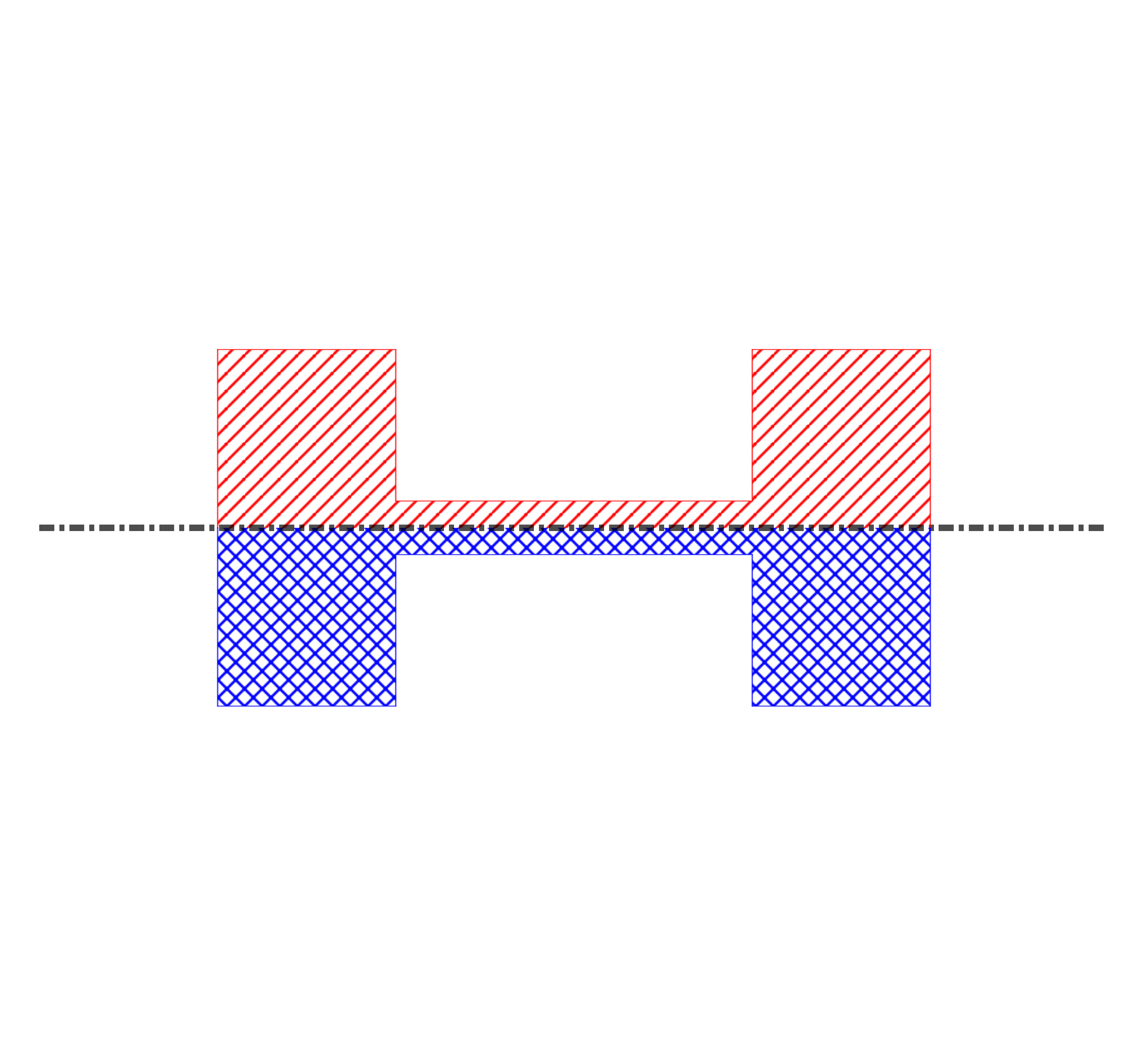}
          \caption{}
          \end{subfigure}
          \caption{Two partitioning strategies
            for the dumbbell shaped domain of Example 2. The partition
            (A)  leads to
            well-separated clusters in the limit $\vartheta\to 0$, while the partition (B)
            does not. Clearly, the subsets in (B) can be further split into 
          meaningful components.}
          \label{fig:dumbell-partition}
        \end{figure}

\subsection{Extensions and generalizations of main results}  

As outlined in Section~\ref{sec:Proofs} the proofs of our main results are divided into two main parts.

\begin{enumerate}
\item We prove that the measure $F_{\sharp} \nu$ has an orthogonal cone structure provided the model is well-separated.
\item We show that after an orthogonal transformation, the measures $F_{\sharp} \nu$ and $F_{n \sharp} \nu_n$ are close to each other in the Wasserstein sense with high probability.
\end{enumerate}

We  emphasize that for part (1) we do not need to assume that $\M$ has compact support, nor that the density $\rho$ is bounded bellow and above by positive constants. We only need Assumptions \ref{Assump1}.

In contrast, for  part (2), we impose the extra conditions in Assumptions \ref{assumption-on-rho-contrast}  to guarantee that  the results of \cite{BIK,Moritz} on the spectral convergence of graph Laplacians are applicable. We further remark that any improvement or extension of results like those in \cite{BIK,Moritz} can be directly incorporated in the analysis that we present in this paper. For the sake of concreteness, here are some interesting directions to consider.

\subsubsection{ Extensions to unbounded domains and $k$-NN graphs}
\label{UnboundedDomains}

In this paper Theorem \ref{TheoremTranps} links the graph Laplacian
$\Delta_n$ to the differential operator $\Delta_\rho$.  Assumption \ref{assumption-on-rho-contrast} is only used when proving Theorem \ref{TheoremTranps}, and in particular to show the existence of transport maps $T_n$ between $\nu$ and $\nu_n$ satisfying 
\[ \lVert T_n - Id \rVert_{\infty}:= \esssup_{x \in \M}| T_n(x) -x |  \rightarrow 0,\]
at an explicit rate in terms of $n$. For an unbounded manifold $\M$, \textit{any} transport map $T_n$ pushing forward $\nu$ into $\nu_n$, must transport mass over arbitrary long distances (even a tiny bit of mass transported over a large distance already results in a large $\infty$-OT distance). It is therefore of relevance to investigate other forms of establishing the discrete-continuum link that do not involve using $\infty$-OT maps; this is the topic of current investigation.

It is also of relevance to extend our analysis to settings where different graph constructions are used to extract the coarse structure of the point cloud.  In this paper we have focused on proximity graphs constructed using a kernel $\eta$ and a connectivity parameter $\veps>0$, but other graph constructions are perhaps more popular among practitioners and enjoy higher regularity properties; such is the case of $k$ nearest neighbors ($k$-NN) graphs. Results quantifying spectral convergence rates for $k$-NN graph Laplacians are largely missing in the literature, and it  is interesting to extend the analysis showed in this paper to the $k$-NN setting. In \cite{kNNNGT}, variational techniques have been used to study the statistical consistency of optimization problems defined on $k$-NN graphs.

\subsubsection{Different weighted versions of the Laplacian}
\label{sec:differentweights}

The kernelized Laplacian considered in this paper is just one of many operators that  can be used to
construct the embedding $F$. In light of the notion of \textit{diffusion maps} introduced in \cite{Coifman1}, it seems reasonable to ask what is the effect of different graph Laplacian
normalizations on the geometry of the embedded data. For example, one may consider a family of renormalizations, where for a given $\gamma \in [0,1] $, we let
\begin{equation*}
 (\W_n)_{ij} :=  \frac{\eta_\veps(|\x_i-\x_j|)}{\tilde d_i^\gamma \tilde d_j^\gamma }, \quad \tilde d_i:= \sum_{j=1}^n \eta_\veps(|\x_i-\x_j|)\,.\footnote{Formally, one could take any $\gamma \in \R$, but we  follow here the choice of $\gamma$ as  in \cite{Coifman1}.} 
\label{eqn:CoiffmanWeights}
 \end{equation*}
The weights $(\W_n)_{ij}$ induce the family of normalized graph Laplacians 
\[ \Delta_n^\gamma:= I - \D_n^{-1} \W_n\,\] 
where $\D_n=diag(d_i)$, $d_i=\sum_j (\W_n)_{ij}$, denotes the degree matrix of the weight matrix $\W_n$.

Intuitively, the parameter $\gamma$  controls the effect of the density $\rho$ on
the clustering: for $\gamma$ close to one, the effect of $\rho$ is minimal, and the limiting clustering is influenced only by the geometry of $\M$; on the other hand, when $\gamma$ is close to zero, the effect of the ground-truth density $\rho$ is maximal (see \cite{Coifman1}). This observation can be made rigorous using the ideas presented in the work \cite{GTSSpectralClustering} which when applied would show that after appropriate rescaling, the spectrum of the graph Laplacian $\Delta^\gamma_n$ approximates that of the formally defined differential operator 
\[\Delta_\rho^\gamma u:= -\frac{1}{\rho^{2(1-\gamma)}} \text{div}( \rho^{2(1-\gamma)}\nabla u ). \]
Furthermore, following the discussion in Section \ref{App:A} one can bootstrap the ideas presented in \cite{Moritz} and obtain quantitative error estimates relating the spectra of $\Delta_n^\gamma$ and $\Delta_\rho^\gamma$ analogue to the ones that appear in Theorem \ref{TheoremTranps}. We notice that the limiting differential operator $\Delta_\rho^\gamma$ is of the same type as the one we have considered in this paper, but where now we think of the underlying density function as being proportional to $\rho^{2(1-\gamma)}$. That is, we can introduce the density 
	\[\widetilde \rho (x):= \frac{\rho^\theta(x)}{\int_\M\rho^\theta(y)\,dy},\]
	and notice that \[ \Delta_\rho^{\gamma}=\Delta_{\tilde\rho},\]  for $\theta=2(1-\gamma)$. Based on the analysis presented in this paper, it is then possible to describe the geometry of the graph Laplacian embeddings constructed from $\Delta^\gamma_n$, assuming that the modified density $\widetilde{\rho}$ admits a representation as a well-separated mixture model.
	 
	An immediate question that arises is to understand whether there is an interesting criterion under which one can choose the value of $\gamma$ that produces the ``best" clustering for a given data set. Alternatively, instead of looking at a single value of $\gamma$, it maybe worth working with the full ensemble of Laplacian embeddings (for all values of $\gamma$), and determine how to use it as best as possible. The formulation and analysis of these questions are the topic of current investigation.

\begin{remark}
Note that the so called \textit{random walk} graph Laplacian \cite{art:ShiMalik00NCut} coincides with $\Delta_n^{0}$ (i.e.  $\gamma=0$). The limiting differential operator then takes the form
\[ \Delta_\rho^0 = - \frac{1}{\rho^2}\divergence(\rho^2 \nabla u ).   \] 
\end{remark}

\begin{remark}
Asymptotically, the kernelized graph Laplacian $\Delta_n$ studied here coincides with $\Delta_n^{1/2}$ (i.e.  $\gamma=1/2$). In particular, there is no reason to work with $\Delta^{1/2}_n$ (i.e. renormalizing twice).  In addition, the  limiting differential operator $\Delta_\rho^{1/2}= \Delta_\rho$ has the following properties. First,
\[ \langle  \Delta_\rho u , v \rangle_\rho= \int_{\M} \nabla u \cdot \nabla v \rho dx, \]
is linear in $\rho$. Secondly,
\[ \langle u ,v \rangle_\rho= \int_{\M} uv \rho dx, \] 
is also linear in $\rho$, and $\Delta_\rho$ is self-adjoint with respect to $\langle \cdot, \cdot \rangle_\rho$. Thirdly, under fairly general conditions, such as when $\rho$ is
a well-separated mixture, $\Delta_\rho$ has a spectral gap after the $N$-th eigenvalue.

% when the density $\rho$ is allowed to vanish (as we imagine is the case for the components $\rho_k$ of our mixture model).%  such as 
% \[ -\frac{1}{\rho^a}\divergence( \rho^b \nabla u), \]
% where $a <b$ \cite{FHBHAS}. When  $a=1$ and $b=2$ this operator is the continuum limit of the unnormalized graph Laplacian \cite{GTSSpectralClustering}. 
\end{remark}

%----------------------------------------------------------------------
\section{Proofs of main results}
\label{sec:Proofs}
%-------------------------------------------------------------------------

%-----------------------------------------

\subsection{Stability of orthogonal cone structures.}
\label{sec:Conic}

In this section we prove some basic results on the stability of orthogonal cone structures. The first result follows immediately from the characterization of weak convergence of probability measures via Portmanteau's theorem \cite[Theorem~8.4.7]{bogachev1}.
\begin{proposition}
\label{Prop:conestruc1}
Let $\left\{ \mu_n \right\}_{n \in \N}$ be a sequence of Borel probability measures on $\R^k$ that converge weakly to a measure $\mu \in \mathcal{P}(\R^k)$. If $\mu $ has an orthogonal cone structure with parameters $(\sigma, \delta,r)$, then for every $\delta >\eps>0$, there exists $K \in \N$ such that for all $n \geq K$ , $\mu_n$ has a orthogonal cone structure with
parameters $(\sigma, \delta + \eps, r)$.
\end{proposition}
\begin{proof}
Consider the open set $U:= \cup_{j=1}^{k}C(e_j,\sigma, r )$, where the sets $C(e_j,\sigma,r)$ are as in Definition~\ref{def:contcones} on the orthogonal cone structure property for $\mu$. Since $\mu_n \converges{w} \mu$, Portmanteau's theorem implies that
\[ \liminf_{n \rightarrow \infty} \mu_n( U) \geq \mu(U). \]
In particular, this means that for every $\delta>\veps>0$, there exists a $K$ such that if $n \geq K$ then,
\[  \mu_n(U) \geq \mu(U) -\veps \geq 1-\delta - \veps.\]
\end{proof}

An immediate corollary of the previous proposition is the following.

\begin{corollary}\label{prop:contdiscrcones}
If a probability measure $\mu$ has an orthogonal cone structure with parameters $(\sigma, \delta, r)$ and $\mu_n$ is the empirical measure associated to $n$ i.i.d. samples from $\mu$, then with probability one, there exists $K=K(\veps, \delta) \in \N$ such that for every $n \geq K$, $\mu_n$ has a $(\sigma, \delta + \veps, r)$-cone structure. 
\end{corollary}

The idea is now to improve the previous asymptotic result and instead study the stability of orthogonal cone structures under small perturbations of a measure; we use the Wasserstein distance to measure the distance between different probability measures. We recall its definition.

\begin{definition}
Let $\mu_1, \mu_2$ be two probability measures on $\R^k$ with finite second moments. We define their Wasserstein distance by 
\[ (W_{2}(\mu_1, \mu_2))^2 := \min_{\pi \in \Gamma(\mu_1, \mu_2)}  \int_{\R^k \times \R^k} |x-y|^2 d \pi(x,y),   \]
where $\Gamma(\mu_1, \mu_2)$ stands for the set of transportation plans between $\mu_1$ and $\mu_2$, that is, the set of probability measures in $\mathcal{P}(\R^k \times \R^k)$ with first and second marginals 
equal to $\mu_1$ and $\mu_2$ respectively. 
\end{definition}

\begin{proposition}
Let $\mu_1, \mu_2  \in \mathcal{P}(\R^k)$ and suppose that $\mu_1$ has an orthogonal cone structure with parameters $(\sigma, \delta, r)$, where $\sigma < \pi/4$. Let $s,t>0$ be such that
\begin{equation}
\frac{r t \sin(s) }{\sqrt{k}}  \geq W_2(\mu_1, \mu_2),  
\label{Relab}
\end{equation}
and such that $\sigma + s < \pi /4$. Then, $\mu_2$ has an orthogonal cone structure with parameters $(\sigma +s  , \delta +t^2 , r(1-\sin(s))$. 
\label{LemmaConeWasserstein}
\end{proposition}

\begin{proof}
  Denote by $C_1, \dots, C_k$ the ``cones" associated to $\mu_1$ and let $e_1, \dots, e_k$ be the orthonormal vectors they are centered at. For $s, t>0$ satisfying \eqref{Relab} and
  such that $s+ \sigma < \pi/4$, define  
\[  \tilde{C}_j := \left\{ z \in \R^k \:  : \:        \frac{ z \cdot e_j}{|z|} > \cos(\tilde{\sigma})  , \quad |z| > \tilde{r}       \right\},\]
where $\tilde{\sigma} := \sigma + s$ and $\tilde{r} := r(1-\sin(s))$.  
Let $\pi \in \Gamma(\mu_2, \mu_1)$ be an optimal transportation plan between $\mu_2$ and $\mu_1$ (which exists by \cite[Proposition 2.1]{Villani}). That is, 
\[  \int_{\R^k \times \R^k} |x-y|^2 d \pi(x,y) = \left(W_2(\mu_1, \mu_2)\right)^2 .\]
First, observe that from the definition of $\tilde{r}$, simple trigonometry shows that 
\[ \min_{ v \in C_k , w \in \partial \tilde{C}_k   } |v - w|  = r \sin(s). \]
From this, it follows that 
\[   r^2 \sin^2(s) \pi(C_j,\R^k \setminus \tilde{C}_j)    \leq \int_{C_j\times (\R^k \setminus \tilde{C}_j)}  |x-y|^2 d\pi(x,y) \leq   \left(W_2(\mu_1, \mu_2)\right)^2.   \]
Hence, for every $j=1,\dots,k$,
\begin{align*}
\begin{split}
\mu_2 (\tilde{C}_j) & \geq \pi(C_j,\tilde{C}_j)
\\& = \pi(C_j,\R^k) - \pi(C_j, \R^k \setminus \tilde{C}_j )
\\& \geq \mu_1(C_j) - \frac{\left( W_2(\mu_1, \mu_2) \right)^2}{r^2 \sin^2(s)}
\\& \geq \mu_1(C_j)  - \frac{t^2}{k}.
\end{split}
\end{align*}
The assumption $\sigma < \pi/4$ and $\sigma+s<\pi/4$ implies that the sets $C_1,\dots, C_k$ and
respectively $\tilde C_1,\dots, \tilde C_k$ are pairwise disjoint and thus
\[  \mu_2 \left( \bigcup_{j=1}^{k} \tilde{C}_j   \right)  \geq \mu_1\left( \bigcup_{j=1}^{k} C_j \right)  - t^2  \geq 1-(\delta + t^2).  \]
\end{proof}

We finish this section by recalling that a map $T : \R^k \rightarrow \R^k$ with $\mu_2 = T_{\sharp} \mu_1$ (i.e. $\mu_2$ is the pushforward of $\mu_1$ by $T$) is called a \textit{transportation map} between $\mu_1$ and $\mu_2$. Such a map induces a transportation plan $\pi_T \in \Gamma(\mu_1, \mu_2)$ given by
\begin{equation*}
\pi_T := (Id \times T)_{\sharp} \mu_1,
\label{MapsPlans} 
\end{equation*}
where the map $Id \times T$ is given by $x \in \R^k \mapsto (x,T(x)) \in \R^k \times \R^k$. In particular, for any such  $T$, we have
\begin{equation}
(W_2(\mu_1, \mu_2))^2 \leq \int_{\R^k \times \R^k} |x-y|^2 d \pi_T(x,y) = \int_{\R^k} |x-T(x)|^2 d \mu_1(x).
\label{WassersteinOptimal}
\end{equation}

\subsection{Proof of Theorem \ref{mainTheorem}}
\label{sec:ProofsMain}

In what follows we use $u_1, \dots, u_N$ to represent the first $N$ eigenfunctions of $\Delta_{\rho}$ and denote by $U$ their span. Normalized in $L^2(\rho)$, the functions $u_1, \dots, u_N$ form an orthonormal basis for $U$.  We use $\PN: L^2(\rho) \rightarrow U$ to denote the projection onto $U$. Denote by $Q$ the span of the functions $q_1, \dots, q_N$. Our first proposition quantifies how close
are the functions $q_k$ to their projections onto $U$.

 \begin{proposition}
  \label{ProjectionsError}
  For every $k=1, \dots, N$ we have
  \begin{equation}
  \frac{1}{w_k}\lVert q_k - \PN(q_k) \rVert^2_\rho \leq \frac{\mathcal{C}}{\lambda_{N+1}},
  \end{equation}
  where $\PN$ stands for the projection onto $U$ the span of the first $N$ eigenfunctions of $\Delta_{\rho}$ and $\lambda_{N+1}$ is the $(N+1)$-st eigenvalue of $\Delta_{\rho}$, where the
  eigenvalues are assumed to be indexed in increasing order.
 \end{proposition}
 \begin{proof}
  A direct computation using the definition of $q_k$ shows that 
  \[ \langle \Delta_{\rho} q_k , q_k \rangle_{\rho}= \int_{ \M} \lvert \nabla q_k \rvert^2 \rho(x) dx   = w_k\mathcal{C}_k,\]
  where $\C_k$ is as defined in \eqref{Coupling}. On the other hand, we can write $q_k$ in the orthonormal basis of eigenfunctions $\{u_1, u_2, \dots \}$ of $\Delta_{\rho}$ as
  \[ q_k = \sum_{l=1}^\infty a_{lk}u_l,   \]
  for some coefficients $\{a_{lk} \}_{l\in \N}$. Using this representation we can alternatively write $\langle \Delta_{\rho} q_k , q_k \rangle_{\rho}$ as
  \[  \langle \Delta_{\rho} q_k , q_k \rangle_{\rho} = \sum_{l=1}^N a_{lk}^2 \lambda_l + \sum_{l=N+1}^\infty a_{lk}^2 \lambda_{l}.  \] 
  We deduce that
  \[ w_k\mathcal{C} \geq w_k \mathcal{C}_k \geq \lambda_{N+1}\sum_{l=N+1}^\infty a_{lk}^2  =  \lambda_{N+1} \lVert q_k - \PN(q_k) \rVert^2_\rho.\]
 \end{proof}

We now focus on getting lower bounds for $\lambda_{N+1}$ in terms of the parameters of the mixture model. In particular, we show that when $\mathcal{S}$ and $\mathcal{C}$ are small in relation to $\Ind$ then $\lambda_{N+1}$ is large. We start with two preliminary results.

%---------------------------------------------------------------------

\begin{lemma}\label{lem:qjnormbound}
For every $j\in \{1,\dots, N \}$,
\[ 1 - \mathcal{S}\leq  \langle q_j , q_j\rangle_{\rho_j} \leq 1 + \mathcal{S}. \]
\label{normqm}
\end{lemma}
\begin{remark}
It follows from the above lemma that $\|q_j\|_{\rho_j}$ converges to one as $\S \to 0$. In general however, $\S$ may be bigger than one, in which case the lower bound is trivial.
\end{remark}

\begin{proof}
From the definition of $q_j$ we see that
\[ \langle q_j, q_j \rangle_{\rho_j}  = 1 + \int_{\M}\left( \frac{w_j\rho_j}{\rho} - 1 \right) \rho_j dx.\]
On the other hand, 
\begin{align*}
\begin{split}
\left \lvert   \int_{\M}\left( \frac{w_j\rho_j}{\rho} - 1 \right) \rho_j dx \right \rvert  & \leq \int_{\M} | w_j \rho_j - \rho | \frac{\rho_j}{\rho}dx  
\\ & = \int_{\M}\sum_{k \not = j} w_k \frac{\rho_k \rho_j}{\rho} dx
\\ & =  \sum_{k \not = j} w_k \int_{\M} \frac{\rho_k \rho_j}{\rho} dx
\\ & \leq \mathcal{S}.
\end{split}
\end{align*}
From the above computations we obtain the desired inequality.
\end{proof}

 \begin{lemma}
  For every $j\in \left\{1, \dots, N \right\}$,
  \begin{equation}
  \inf_{ \langle  v ,  q_j    \rangle_{\rho_j} =0  } \frac{ \int_{\M} | \nabla v |^2 \rho_j dx}{\langle v, v \rangle_{\rho_j} } \geq \Ind(1- \mathcal{S}),
  \label{eqnLemmaEignvalueC2New}
  \end{equation}
%   where $\Ind$ is the indivisibility parameter defined in \eqref{Lambda-definition} and $\S$ is the overlapping parameter defined in \eqref{eqn:S2}.
  \label{LemmaEignvalueC2New}
 \end{lemma}
 
 \begin{proof}
   Let us fix $j \in \left\{1, \dots, N \right\}$ and pick $v \in  H^1(\M, \rho_j)$,
     such that 
   $\langle  v ,  q_j    \rangle_{\rho_j} =0$, and $\langle v, v \rangle_{\rho_j} =1$, where
   \begin{equation}
     \label{H1-M-definition}
     H^1(\mathcal{M},\rho_j) := \left\{ u \in L^2(\mathcal{M}, \rho_j) | \int_{\mathcal M}
       \left( | \nabla u|^2 + |u|^2\right) \rho_j dx< + \infty \  \right\}.
   \end{equation}
  Observe that
  \[ \int_{\M} | \nabla v |^2 \rho_j dx  = \sum_{k=1}^{\infty} \langle v , e_{j,k} \rangle^2_{\rho_j} 
\lambda_{j,k},      \]
  where in the above $\{\lambda_{j,k}, e_{j,k}\}$  are the orthonormal (w.r.t. $\langle \cdot, \cdot \rangle_{\rho_j}$) eigenpairs of $\Delta_{\rho_j}$. Since $\lambda_{j,1}=0$, the first eigenvector is given by $e_{j,1}\equiv \mathds{1}$, the function which is identically equal to one. Using the above equality and the Pythagorean theorem we conclude that
  \begin{align}
  \int_{\M} | \nabla v |^2 \rho_j dx  \geq \lambda_{j,2} \sum_{k=2}^{\infty} \langle v , e_{j,k} \rangle^2_{\rho_j}  
  &=  \lambda_{j,2} \left( \langle v , v  \rangle_{\rho_j}    - \langle v, e_{j,1}\rangle^2_{\rho_j}   \right)  \notag\\
  &=   \Ind_j \left(1    - \langle v, e_{j,1}\rangle^2_{\rho_j}   \right),
  \label{auxLemmalambda}
  \end{align}
  where we recall $\Ind_j$ was defined in \eqref{Indivisbility}, and corresponds to the second eigenvalue of the operator $\Delta_{\rho_j}$ by the max-min formula \cite[Thm.~8.4.2]{Buttazzo}.
   
 We now find an upper bound for $\langle v, e_{j,1}\rangle^2_{\rho_j}$. Notice that
  \[ \langle v, e_{j,1}\rangle_{\rho_j}  =  \int_{\M} v \rho_j dx = \int_{\M} v( 1 - q_j) \rho_j dx,     \]
  where the second equality follows from the fact that $\langle  v ,  q_j    \rangle_{\rho_j} =0$. The Cauchy-Schwartz inequality implies that,
  \[ \langle v, e_{j,1}\rangle_{\rho_j}^2 \leq \lVert v \rVert_{\rho_j}^2 \int_{\M}(1- q_j )^2 \rho_j dx = \int_{M}  (1-q_j)^2 \rho_j dx = 1 + \int_{\M} (q_j^2 -2q_j ) \rho_j dx.   \]
  From the fact that $0 \leq q_j \leq 1$, we conclude that $(q_j^2 -2q_j ) \leq -q_j^2$, and therefore 
  \[ \langle v, e_{j,1}\rangle_{\rho_j}^2 \leq  1 - \int_{\M} q_j^2 \rho_j dx = 1 - \langle q_j, q_j \rangle_{\rho_j} \leq 1- (1 - \mathcal{S}) = \mathcal{S},   \]
  where the second inequality follows from Lemma \ref{normqm}. To conclude, the estimate \eqref{eqnLemmaEignvalueC2New} is a consequence of the above inequality and \eqref{auxLemmalambda}.
 \end{proof}
 
 We are now in a position to find a lower bound for the $(N+1)$-th eigenvalue $\lambda_{N+1}$ of $\Delta_{\rho}$.
\begin{proposition}[Lower bound for $\lambda_{N+1}$]
\label{auxLemmaSigmaB}
Suppose that $N\S <1$.
 Then, 
 \[ \left( \sqrt{\Ind(1 - N\S)}  - \frac{\sqrt{\C N \S}}{(1- \S)}\right)^2 \leq \lambda_{N+1} \,.\]
\end{proposition}

% With condition $\S(N+1)<1$: In other words, as the number of components $N$ becomes large, we require less overlap between each of the components. This makes sense since the $\rho_k$'s have full support, and so even small tails add up for a large number of mixtures which makes it harder to distinguigh components.

\begin{proof}
Let us consider $u \in H^1(\M, \rho)$ satisfying $\langle u, u \rangle_{\rho}=1$, and $u\in Q^\perp$, i.e. $u$ is orthogonal to the span of $\{q_1,...,q_N\}$ in terms of the inner product $\langle \cdot , \cdot \rangle_{\rho}$. From the fact that $\langle u,q_j\rangle_\rho=0$ for all $j\in\{1,...,N\}$, we deduce that for every $j=1, \dots, N$,
\[  \langle  u , q_j \rangle_{\rho_j} =  \int_{\M} u q_j \rho_j dx =   \frac{1}{w_j}\int_{\M} u q_j \left( w_j \rho_j- \rho    \right) dx =-  \frac{1}{w_j} \sum_{k \not = j} w_k \int_{\M} u q_j \rho_k dx . \]
Hence, 
\begin{align}
\begin{split}
\lvert  \langle u , q_j \rangle_{\rho_j}   \rvert  
&\leq  \frac{1}{w_j}\sum_{k \not = j } w_k \int_{\M} \lvert u \rvert q_j \rho_k dx
\\ & \leq \frac{1}{w_j}\sum_{k \not = j } w_k \left(  \int_{\M} u^2  \rho_k dx \right)^{1/2} \left(  \int_{\M} q_j^2  \rho_k dx \right)^{1/2}
\\& \leq \sqrt{\frac{\mathcal{S}}{w_j}}\sum_{k \not = j } w_k \left(  \int_{\M} u^2  \rho_k dx \right)^{1/2}
\\& \leq \sqrt{\frac{(1-w_j)\S}{w_j}}\left( \sum_{k \not = j } w_k   \int_{\M} u^2  \rho_k dx \right)^{1/2}
\leq \sqrt{\frac{(1-w_j)\S}{w_j}},
\end{split}
\label{normqmInnerprod}
\end{align}
where in the second inequality we have used the Cauchy-Schwartz inequality, in the fourth inequality we have used Jensen's inequality, and in the last inequality the fact that $\langle u,u \rangle_{\rho}=1$. 
Define now the functions 
\[ v_j:= u - \left(\frac{\langle u, q_j \rangle_{\rho_j}}{ \langle q_j , q_j \rangle_{\rho_j} }\right) q_j, \quad j=1, \dots, N.\]
The function $v_j$ is orthogonal to $q_j$ with respect to the inner product $\langle \cdot,\cdot \rangle_{\rho_j}$. In addition, 
\begin{align*} 
\int_{\M} \lvert \nabla v_j \rvert^2 \rho_jdx  
= &\int_{\M} | \nabla u |^2  \rho_j dx - 2 \frac{\langle u, q_j \rangle_{\rho_j}}{ \langle q_j , q_j \rangle_{\rho_j} } \int_{\M} \nabla q_j \cdot \nabla u \rho_j dx \\
&+ \left( \frac{\langle u, q_j \rangle_{\rho_j}}{ \langle q_j , q_j \rangle_{\rho_j} } \right)^2 \int_{\M} | \nabla q_j |^2 \rho_j dx. 
\end{align*}
Notice that,
\[ \int_{\M} \lvert \nabla q_j \rvert ^2 \rho_j dx  = \frac{w_j}{4}  \int_{\M} \lvert \nabla (\rho_j /\rho) \rvert^2 \rho dx   
= \frac{1}{4}  \int_{\M} q_j^2 \left|\frac{\nabla\rho_j}{\rho_j}-\frac{\nabla \rho}{\rho}\right|^2\rho_j\, dx  \leq\mathcal{C}_j,  
\]
where in the last step we used the fact that $0\leq q_j \leq 1$.  From the above inequality, Lemma \ref{normqm}, and \eqref{normqmInnerprod}, it follows that
\begin{align*}
 \int_{\M} \lvert \nabla v_j \rvert^2 \rho_jdx 
&\leq  \int_{\M} | \nabla u |^2  \rho_j dx  + 
2\left(\sqrt{\frac{1-w_j}{w_j}}\right)\left( \frac{\sqrt{\mathcal{C}\mathcal{S}} }{(1- \mathcal{S}) }  \right) \left( \int_{\M} | \nabla u |^2  \rho_j dx \right)^{1/2}\\
&\quad + \frac{(1-w_j)\mathcal{C}\mathcal{S} }{w_j(1- \mathcal{S})^2},  
\end{align*}
According to Lemma \ref{LemmaEignvalueC2New},  the left hand side of the above expression can be bounded from below by $ \Ind(1 - \mathcal{S}) \langle v_j , v_j \rangle_{\rho_j}$ (since $\langle v_j, q_j \rangle_{\rho_j} =0$). On the other hand, 
\[  \langle v_j , v_j \rangle_{\rho_j}  = \langle  u, u \rangle_{\rho_j} - \frac{ \langle u, q_j \rangle_{\rho_j}^2}{\langle q_j, q_j \rangle_{\rho_j}} \geq \langle  u, u \rangle_{\rho_j} - \frac{(1-w_j)\mathcal{S}}{w_j(1 - \S)},  \]
as it follows from \eqref{normqmInnerprod} and Lemma \ref{normqm}.  
Therefore, 
\begin{equation*}
          \begin{split}
\Ind(1 - \mathcal{S}) &\langle u , u \rangle_{\rho_j} - \frac{(1-w_j)\Ind \mathcal{S}}{w_j}
\\ & \le\int_{\M} | \nabla u |^2  \rho_j dx  + 2\left(\sqrt{\frac{1-w_j}{w_j}}\right)\left( \frac{\sqrt{\mathcal{C}\mathcal{S}} }{(1- \mathcal{S}) }  \right) \left( \int_{\M} | \nabla u |^2  \rho_j dx \right)^{1/2}\\
&\quad + \frac{(1-w_j)\mathcal{C}\mathcal{S} }{w_j(1- \mathcal{S})^2 } .
      \end{split}
\label{AuxSigmaC1}
\end{equation*}
Multiplying both sides of the above inequality by $w_j$ and adding over $j$ we deduce
\[  \Ind(1 - N\mathcal{S})  \leq   \int_{\M} | \nabla u |^2  \rho dx  
+  2\left( \frac{\sqrt{\mathcal{C}\mathcal{S}} }{(1- \mathcal{S}) }  \right) \sum_{j=1}^{N}\left( \int_{\M} | \nabla u |^2  w_j\rho_j dx \right)^{1/2} +  \frac{\mathcal{C} N \mathcal{S}}{(1- \mathcal{S})^2} . \]
Applying Jensen's inequality and using the fact that $\langle u,u \rangle_\rho=1$ we have
\begin{align*}
\begin{split}
\Ind(1 - N\mathcal{S})& \leq  \int_{\M} | \nabla u |^2  \rho dx  
% + \frac{\mathcal{C} N \mathcal{S}}{(1- \mathcal{S})^2} 
+  2 \frac{\sqrt{\mathcal{C}N\mathcal{S}} }{(1- \mathcal{S}) }   \left( \int_{\M} | \nabla u |^2  \rho dx \right)^{1/2} +  \frac{\mathcal{C} N \mathcal{S}}{(1- \mathcal{S})^2}\\
& = \left(   \lVert \nabla u \rVert_\rho + \frac{\sqrt{ \C N \S} }{(1-\S)}  \right)^2\,.
\end{split}
\end{align*}
That is,
\[  \sqrt{\Ind(1 - N\S)   } - \frac{\sqrt{\C N \S}}{(1- \S)} \leq \lVert \nabla u  \rVert_\rho\,. \]
Given that the above inequality holds for every $u \in H^1(\M,\rho)$ with $u \in Q^\perp$ and $\langle u,u  \rangle_\rho=1$, we conclude from the fact that $Q$ has dimension $N$ and from the max-min formula \cite[Theorem 8.4.2]{Buttazzo},
\[   \left( \sqrt{\Ind(1 - N\S)  } - \frac{\sqrt{\C N \S}}{(1- \S)} \right)^2  \leq  \min_{u \in Q^\perp}    \frac{\int_{\M} |\nabla u|^2 \rho dx}{\int_{\M} u^2 \rho dx }  \leq \lambda_{N+1}  \]
\end{proof}
%-------------------------------------------------------------------
 
 Combining the previous proposition with Proposition \ref{ProjectionsError} we obtain the following result.
 \begin{corollary}
  \label{corProjec}
  For every $k=1, \dots, N$ we have
  \begin{equation} \label{qproj}
  \frac{1}{w_k} \lVert q_k - \PN(q_k) \rVert^2_\rho \leq  \left( 
\sqrt{\frac{\Ind(1 -N\S)}{\C}  } - \frac{\sqrt{ N \S}}{(1- \S)} \right)^{-2}.
  \end{equation}
 \end{corollary}
 The above corollary shows that for a well-separated mixture model, that is, for $\S$ and $\C/\Ind$ small, the right-hand side in \eqref{qproj} is small also, and so in that case, the functions $q_1,...,q_N$ are close to their projections onto the first $N$ eigenfunctions of $\Delta_\rho$. We will be able to conclude that $F_{\sharp}\nu$ has an orthogonal cone structure provided we can show $F^Q_{\sharp}\nu$ has one. 
 
\begin{proposition}\label{prop:muQcone}
 The probability measure $\mu^Q:=F^Q_{\sharp} \nu$ with $F^Q$ defined in \eqref{FQ} has an orthogonal cone structure with parameters $(\sigma,\delta, r)$ for any $\sigma\in(0,\pi/4)$, $\delta^*\leq \delta <1$ and $r=w_{max}^{-1/2}$ where
 $$
 \delta^*:= \frac{w_{max} \cos^2(\sigma) N^2 \S }{w_{min}(1-\cos^2(\sigma))}\,.
 $$
 Here, $w_{max}:=\max_{i  =1, \dots, N} w_i$ and $w_{min}:=\min_{i  =1, \dots, N} w_i$.
\end{proposition}

\begin{proof}
For each $k=1, \dots, N$, let
\[ C_k := \left\{  z \in \R^N \: : \:  \frac{z_k }{|z|}   > \cos(\sigma) , \quad \lvert z \rvert \geq r \right\}, \]
with $r:= \frac{1}{\sqrt{w_{max}}} $ and fixed $\sigma \in (0, \pi/4)$.
It follows that 
\[\mu^Q(C_k) = F^{Q}_{\sharp}\nu(C_k)=  \nu (A_k),\]
where $A_k$ is the preimage of $C_k$ through $F^Q$, i.e.
\begin{align*}
 A_k:= &(F^Q)^{-1}(C_k) \\
 =&\left\{  x \in \M   \: : \:  \sqrt{\frac{\rho_k(x)}{\rho(x)}} > \cos(\sigma) \left( \sum_{j=1}^{N} \frac{\rho_j(x)}{\rho(x)}\right)^{1/2},\: \left( \sum_{j=1}^{N} \frac{\rho_j(x)}{\rho(x)}\right)^{1/2} > r  \right\}.
\end{align*}
From the definition of $r$ we see that the condition 
\[  \left( \sum_{j=1}^{N} \frac{\rho_j(x)}{\rho(x)}\right)^{1/2} \geq r \]
is redundant and so we can write 
\[ A_k=\left\{  x \in \M   \: : \:  \rho_k(x) > \cos^2(\sigma)  \sum_{j=1}^{N} \rho_j(x) \right\}.  \]

%
%Now, let $\tilde{A}_1, \dots, \tilde{A}_N$ be the sets 
%\[ \tilde{A}_i: = \left\{ x \in \M \: : \: \frac{w_i \rho_i(x)}{\rho(x)} \geq \cos^2(\sigma) \frac{w_{max}}{w_{min}} \right\}, \quad i=1, \dots, N. \]
%A simple computation shows that for every $i$ we have $\tilde{A}_i \subseteq A_i$, and so
%\begin{equation} \mu^{Q} \left( \cup_{i=1}^N C_i \right) = \nu\left( \cup_{i=1}^N A_i \right) \geq  \nu\left( \cup_{i=1}^N \tilde A_i \right).  
%\label{auxcones}
%\end{equation}

Now, for an arbitrary  $x_0\in \cap_{j=1}^N  A_j^c \subseteq \M$ we have
\[    \rho_k(x_0) \leq  \cos^2(\sigma) \sum_{j=1}^N \rho_j(x_0) , \quad \forall k=1, \dots, N, \]
or what is the same, 
\[  (1- \cos^2(\sigma))\rho_k(x_0) \leq  \cos^2(\sigma) \sum_{j \not = k} \rho_j(x_0), \quad \forall k =1, \dots, N.   \]
From the fact that, 
\[ \sum_{k=1}^N\frac{w_k \rho_k(x_0)}{\rho(x_0)} =1, \]
we know there exists a $\hat{k } \in \{ 1, \dots, N\}$ (depending on $x_0$) for which 
\[  \frac{w_{\hat k} \rho_{\hat k}(x_0)}{\rho(x_0)} \geq \frac{1}{N}. \]
Hence, 
\begin{align*}
 \frac{1- \cos^2(\sigma)}{N^2} \leq  (1- \cos^2(\sigma)  ) \left(\frac{w_{\hat k} \rho_{\hat k}(x_0)}{\rho(x_0)}\right)^2 &\leq\cos^2(\sigma)\frac{w_{max}}{w_{min}} \sum_{j \not = \hat{k}} \frac{w_j \rho_j(x_0)}{\rho(x_0)} \frac{w_{\hat k} \rho_{\hat k}(x_0)}{\rho(x_0)} 
 \\& \leq \cos^2(\sigma)\frac{w_{max}}{w_{min}}\sum_{k} \sum_{j\not = k} \frac{w_j \rho_j(x_0)}{\rho(x_0)} \frac{w_{k} \rho_{k}(x_0)}{\rho(x_0)}.
 \end{align*}
%and we claim that for any such $x$ there exists $i,j$ with $i\not =j$ for which 
%\[ \frac{w_i \rho_i(x_0) }{\rho(x_0)} \geq \frac{1}{N} , \quad  \frac{w_j \rho_j(x_0) }{\rho(x_0)} > \frac{1-\cos^2(\sigma)\cdot (w_{max}/w_{min})}{N-1}.\]
%To see this, let $i$ be the index for which $ \frac{w_i \rho_i(x_0)}{\rho(x_0)} $ is maximum. Since
%\[ \sum_{j=1}^N  \frac{w_j \rho_j(x_0)}{\rho(x_0)} =1 \]
%we must necessarily have $\frac{w_i \rho_i(x_0)}{\rho(x_0)} \geq \frac{1}{N}$. On the other hand,
%\[ \sum_{j\not = i}\frac{w_j \rho_j(x_0) }{\rho(x_0)} = 1 - \frac{w_i \rho_i(x_0)}{\rho(x_0)} > 1-\cos^2(\sigma)\cdot (w_{max}/w_{min}) , \]
%from where it follows that at least for one $j\not = i$, $w_j\rho_j/\rho$ must be strictly larger than $\frac{1-\cos^2(\sigma)\cdot (w_{max}/w_{min})}{N-1} $. 
Since this is true for every $x_0 \in \cap_{k=1}^N A_k^c$, we conclude that
\begin{align*}
 \frac{1-\cos^2(\sigma)}{N^2} \nu \left( \cap_{k=1}^N  A_k^c \right) 
 &\leq \cos^2(\sigma)\frac{w_{max}}{w_{min}} \int_{ \cap_{k=1}^N   A_k^c}  \left(\sum_{l=1}^N\sum_{j\not = l} \frac{w_l\rho_l }{\rho} \frac{w_j\rho_j }{\rho}\right) \rho dx \\
 &\leq \cos^2(\sigma)\frac{w_{max}}{w_{min}} \sum_{l=1}^N\sum_{j\not = l} w_l w_j \int_{\M}   \frac{\rho_l \rho_j }{\rho}  dx 
 \leq  \cos^2(\sigma)\frac{w_{max}}{w_{min}} \S.
\end{align*}
That is, 
\[  \nu \left( \cap_{k=1}^N \tilde A_k^c \right)  \leq \frac{w_{max} \cos^2(\sigma) N^2 \S }{w_{min}(1-\cos^2(\sigma))},  \]
and so we deduce that
\[ \mu^Q\left(  \cup _{k=1}^N C_k   \right)\geq 1- \frac{w_{max} \cos^2(\sigma) N^2 \S }{w_{min}(1-\cos^2(\sigma))}\,.  \]
This concludes the proof.
\end{proof}

\begin{proof}[Proof of Theorem \ref{mainTheorem}]
In order to show that the measure $\mu:=F_{\sharp } \nu$ has an orthogonal cone structure, it is enough to show that the measure $(O F)_{\sharp} \nu$ has an orthogonal cone structure where $OF$  is the map:
\[  OF : x \in \M \mapsto O F(x) \in \R^N,   \]
and where $O$ is a conveniently chosen $N \times N$ orthogonal matrix. We will construct $O$ in such a way that the measures $OF_{\sharp} \nu$ and $F^Q_{\sharp} \nu$ are close to each other in the Wasserstein sense. From this, Proposition~\ref{LemmaConeWasserstein} and Proposition~\ref{prop:muQcone} we will be able to conclude that $OF_{\sharp} \nu$ has an orthogonal cone structure. 
% Note that the restriction $\sigma<\pi/4$ is required for Proposition~\ref{LemmaConeWasserstein}.

Let us introduce 
\[ v_i:= \frac{\PN(q_i)}{ \lVert \PN(q_i) \rVert_\rho}, \quad i=1, \dots, N,\]
where we recall $\PN: L^2(\rho)\rightarrow U$ is the orthogonal projection onto $U$, the span of the first $N$ eigenfunctions of $\Delta_{\rho}$. From Corollary \ref{corProjec} it follows that
\begin{align}
\begin{split}
  \left \lVert \frac{q_i}{\sqrt{w_i}} - v_i \right \rVert_\rho & \leq \left \lVert \frac{q_i}{\sqrt{w_i}} - \frac{\PN(q_i)}{\sqrt{w_i}} \right \rVert_\rho + \left \lVert \frac{
      \PN(q_i)}{\sqrt{w_i}} - \frac{\PN(q_i)}{\lVert \PN(q_i) \rVert_\rho} \right \rVert_\rho
\\& =  \left \lVert \frac{q_i}{\sqrt{w_i}} - \frac{\PN(q_i)}{\sqrt{w_i}} \right \rVert_\rho  + \frac{1}{\sqrt{w_i}} \left \lvert \lVert \PN(q_i) \rVert_{\rho} -\sqrt{w_i}  \right \rvert
\\&=  \left \lVert \frac{q_i}{\sqrt{w_i}} - \frac{\PN(q_i)}{\sqrt{w_i}} \right \rVert_\rho  + \frac{1}{\sqrt{w_i}} \left \lvert \lVert \PN(q_i) \rVert_{\rho} -  \lVert q_i \rVert_{\rho} \right \rvert
\\& \leq 2 \left \lVert \frac{q_i}{\sqrt{w_i}} - \frac{\PN(q_i)}{\sqrt{w_i}} \right \rVert_\rho 
\\& \leq 2  \left( 
\sqrt{\frac{\Ind(1 -N\S)}{\C}  } - \frac{\sqrt{ N \S}}{(1- \S)} \right)^{-1}.
\end{split}
\end{align}
In particular, for $i \not =j$
\begin{align*}
\begin{split}
\lvert \langle v_i , v_j \rangle_{\rho} \rvert & = \left \lvert \langle v_i- \frac{q_i}{\sqrt{w_i}} , v_j \rangle_{\rho} + \langle \frac{q_i}{\sqrt{w_i}}  , v_j- \frac{q_j}{\sqrt{w_j}} \rangle_{\rho} + \langle \frac{q_i}{\sqrt{w_i}} ,\frac{q_j}{\sqrt{w_j}} \rangle_{\rho} \right \rvert
\\ & \leq \left\lVert \frac{q_i}{\sqrt{w_i}} - v_i \right\rVert_\rho + \left\lVert \frac{q_j}{\sqrt{w_j}} - v_j \right\rVert_\rho + \mathcal{S}^{1/2}
\\& \leq 4  \left( 
\sqrt{\frac{\Ind(1 -N\S)}{\C}  } - \frac{\sqrt{ N \S}}{(1- \S)} \right)^{-1} +  \S^{1/2}= \tau.
\end{split}
\end{align*} 
The first inequality follows from an application of the Cauchy-Schwartz inequality for the first two terms (since $\lVert v_i \rVert_\rho=1$ by definition) and Jensen's inequality for the last term. By assumption \eqref{Ndelta<1}, we can then use Lemma \ref{BH-finite-dim-orht-basis}  to conclude that there exists $\tilde{v}_1, \dots, \tilde{v}_N$, an orthonormal basis for $\left(U,\langle \cdot,\cdot\rangle_\rho\right)$ for which 
\[ \lVert v_i - \tilde{v_i} \rVert_\rho^2 \leq N \left(\frac{1}{\sqrt{1-N \tau}} -1 \right)^{2}, \quad i=1, \dots, N. \]
Therefore, for any $i=1, \dots, N$,
\begin{align}
\left\lVert \frac{q_i}{\sqrt{w_i}} - \tilde{v}_i \right\rVert_\rho^2 
% =&\left\lVert \frac{q_i}{\sqrt{w_i}} - v_i \right\rVert_\rho^2 
% +2\langle v_i-\tilde{v}_i,\frac{q_i}{\sqrt{w_i}}\rangle_\rho\notag + \langle \frac{q_i}{\sqrt{w_i}}- v_i , v_i - \tilde v_i \rangle\\
=&\left\lVert \frac{q_i}{\sqrt{w_i}} - v_i \right\rVert_\rho^2 
+2\langle v_i-\tilde{v}_i,\frac{q_i}{\sqrt{w_i}}\rangle_\rho\notag 
- \langle v_i+\tilde v_i , v_i - \tilde v_i \rangle_\rho\\
\leq&\left\lVert \frac{q_i}{\sqrt{w_i}} - v_i \right\rVert_\rho^2 
+4\left\lVert v_i-\tilde{v}_i\right\rVert_\rho \notag\\
\leq & 4 \left( 
\sqrt{\frac{\Ind(1 -N\S)}{\C}  } - \frac{\sqrt{ N \S}}{(1- \S)} \right)^{-2}
+ 4\sqrt{N} \left(\frac{1}{\sqrt{1-N \tau}} -1 \right)\notag\\
= & \left(\frac{\tau-\sqrt{\S}}{2}\right)^2
+ 4\sqrt{N} \left(\frac{1}{\sqrt{1-N \tau}} -1 \right)\,. 
\label{ProofEstimate}
\end{align}
Let $\tilde{F}: \M \mapsto \R^N$ be the map $\tilde{F}(x)=\sum_{j=1}^N\tilde{v}_j(x) e_j$.
From the fact that $\{\tilde{v}_1, \dots, \tilde{v}_N\}$ and $\{u_1, \dots, u_N \}$ are both orthonormal bases for $\left(U,\langle \cdot,\cdot\rangle_\rho\right)$ we can conclude there exists an orthogonal matrix $O \in \R^N \times \R^N$ for which
\[ OF = \tilde{F}. \] 
Moreover, let $\pi:=(F^Q \times \tilde F)_{\sharp} \nu$. We notice that $\pi$ is a coupling between $F^{Q}_{\sharp } \nu$ and $\tilde F_{\sharp} \nu$. From \eqref{ProofEstimate} we see that
\begin{align}
W^2_2(F^Q_{\sharp}\nu, \tilde{F}_{\sharp} \nu ) & \leq \int_{\R^N}\int_{\R^N} \lvert z- \tilde{z} \rvert^2 d \pi(z, \tilde{z}) \notag
\\ & = \int_{\M} \lvert F^Q(x) - \tilde F(x) \rvert^2 d \nu(x) \notag
\\ & =   \sum_{i=1}^N \left \lVert \frac{q_i}{\sqrt{w}_i} - \tilde{v}_i \right \rVert^2_{\rho}\notag
\\ & =N\left(\frac{\tau-\sqrt{\S}}{2}\right)^2
+ 4N^{3/2} \left(\frac{1}{\sqrt{1-N \tau}} -1 \right)\,.\label{estFQFtilda}
\end{align}
It then follows from Proposition~\ref{LemmaConeWasserstein} and Proposition~\ref{prop:muQcone} that $\mu$ has an orthogonal cone structure with parameters $\left(\sigma+s,\delta+t^2, \frac{1-\sin(s)}{\sqrt{w_{max}}}\right)$ with $\delta\in[\delta^\ast,1)$ as given in Proposition~\ref{prop:muQcone} for any $s,t>0$ satisfying assumption \eqref{muconehyp}.
\end{proof}

%%%%%%%%%%%%%%%%%%%%%%%%%%%%%%%%%%%%%%%%%%%%%%%%%%%%%%%%%%%%%%
\subsection{From discrete to continuum: convergence results for the spectrum of $\Delta_n$}
\label{App:A}

In this section we connect the spectrum of the graph Laplacian $\Delta_n$ with the spectrum of  $\Delta_\rho$. For this purpose, we modify our notation slightly and
write $L^2(\nu)$ for $L^2(\rho)$ to allow for comparison with $L^2(\nu_n)$ (recall \eqref{Empirical}). Furthermore,
we write $\langle u, v \rangle_{L^2(\mu)} := \int_\M u v \: d\mu(x)$ for
 $\mu = \nu$ or $\nu_n$.

\begin{theorem}
\label{TheoremTranps}
Let $\M$, $\rho$, $\veps$ and $n$ be as in Assumptions \ref{assumption-on-rho-contrast}. Let $u_{n,1}, \dots, u_{n,N}$ be the first $N$ eigenvectors of the kernelized graph Laplacian $\Delta_n$ (i.e. the graph with weights defined in \eqref{KernelizedWeights}).
Then, for every $\beta>1$ there exists a constant $C_\beta>0$ such that with probability at least $1- C_\beta n^{-\beta}$, there exists a map $T_n : \M \rightarrow \left\{\x_1, \dots, \x_n \right\}$ satisfying
\begin{equation}
\nu(T_n^{-1}(\{ \x_i \})) = \frac{1}{n}, \quad \forall i=1, \dots, n, 
\label{TranspEmpi}
\end{equation}
and
\begin{align}
  \| g_j  - u_{n,j}\circ T_n\|_{L^2(\nu)}^2 
  &\leq c_\M \left( \frac{\lambda_N}{\lambda_{N+1} - \lambda_N} \right)
  \left( \veps + \frac{\log(n)^{p_m}}{\veps n^{1/m}} \right)\,,
  \label{thm:erroreigen}
\end{align}
  for some  orthonormal functions $g_1, \dots, g_N \in L^2(\nu)$ belonging to $U$, the span of the first $N$ eigenfunctions of $\Delta_\rho$ with respect to $\langle\cdot,\cdot\rangle_\rho$,
  and a constant $c_\M>0$ depending only on $\M$, $N$, $\rho^\pm$, $C_{Lip}$ and $\eta$. In the above, $\lambda_N$ and $\lambda_{N+1}$ denote the $N$-th and the $(N+1)$-th eigenvalues of $\Delta_\rho$.
\end{theorem}

\begin{remark}
  Relation \eqref{TranspEmpi} is simply saying that the map $T_n$ is a transportation map from $\nu$ to $\nu_n$. $T_n$ here is as in \cite{Moritz}, i.e. the $\infty$-optimal transport ($\infty$-OT) map between $\nu$ and $\nu_n$. It was proved in \cite{Moritz} that the $\infty$-OT cost scales like $\frac{\log(n)^{p_m}}{ n^{1/m}}$. Notice that this is the only term on the right hand side of \eqref{thm:erroreigen} that explicitly depends on $n$.  
\end{remark}

The proof of Theorem~\ref{TheoremTranps} relies on a comparison between the Dirichlet forms associated to the operators $\Delta_n$ and $ \Delta_\rho$ by way of careful interpolation and discretization of discrete and continuum functions. Only small modifications to the proofs in \cite{BIK,Moritz} are necessary and in what follows we give a thorough explanation of the steps that need to be adjusted. For simplicity, we largely follow the notation in \cite{Moritz}.

Let us summarize our strategy for the proof. The Dirichlet forms associated to $\Delta_n$ and $\Delta_\rho$ are respectively defined as
\begin{equation}
b_n(u_n):=  \frac{1}{2n} \sum_{i,j} \W_{ij}  | u_n(\x_i) - u_n(\x_j) |^2 , \quad u_n \in L^2(\nu_n),
\label{def:bn}
\end{equation}
and
\begin{equation}
D(u):= \frac{1}{2}\int_{\M} | \nabla u |^2 \rho(x) dx    , \quad u \in H^1(\M, \rho),
\label{def:D}
\end{equation}
where we recall $\W_{ij}$ are the weights defined in \eqref{KernelizedWeights}. We will show that 
\begin{align*}
 b_n(Pf) &\leq (1+ er(n, \veps)) D(f) , \quad \forall f \in L^2(\nu)  \,,     \\
 D(If_n) &\leq (1+er(n,\veps)) b_n(f_n), \quad \forall f_n \in L^2(\nu_n). 
\end{align*}
for appropriately defined  {\it discretization} and {\it interpolation} maps $P$ and $I$; the $er(n,\veps)$ term does not depend on $f$ or $f_n$, and is small for $n$ large, and $\veps$ small but larger than a certain quantity that depends on $n$ that we will introduce below.  Together with the properties of the maps $P$ and $I$ (see  Lemma \ref{lem:Dirichlet}), these inequalities allow us to compare the spectra of the discrete and continuous operators.

To define the maps $P$ and $I$ we first introduce the transport maps $T_n$ constructed in
\cite[Theorem~2]{Moritz}.
\begin{proposition}\label{T-n-definition}
  For a given $\beta>1$, there
  exists a constant $C_\beta>0$ depending only on $\beta$ so that with probability at least $1- C_\beta n^{-\beta}$ there exists a map $T_n : \M \rightarrow \{ \x_1, \dots, \x_n \}$ for which:
\begin{enumerate}
\item $\nu(T_n^{-1} (\x_i) ) = 1/n$ for all $i=1, \dots, n$,
\item $\delta_n:= \esssup_{x \in \M} d_\M (T_n(x), x ) \leq C \frac{\log(n)^{p_m}}{n^{1/m}}$, 
\end{enumerate}
where  $C = C( \M, \rho^-, \rho^+, \beta) >0$, $p_m = 3/4$ for $m=2$ and $p_m=1/m$ for $m \geq 3$, and $d_\M$ represents the geodesic distance in $\M$.
\end{proposition}
Given the map $T_n$, the discretization map  
\[P : L^2(\nu) \rightarrow L^2(\nu_n),\] 
is defined as the transformation that takes a given $f \in L^2(\nu)$ and maps it into the discrete function 
\[ Pf(\x_i) :=  n \int_{T_n^{-1}(\{ \x_i\})} f(x)  \rho(x) dx, \quad i=1, \dots, n .  \]
For the interpolation map $I$, we first let $P^*$ be the adjoint of $P$ with respect
to the $L^2(\nu_n)$ inner product, i.e. the map that satisfies
\[ \langle Pg, f_n \rangle_{L^2(\nu_n)} = \langle g, P^\ast f_n \rangle_{L^2(\nu)}, \quad \forall g \in L^2(\nu), \quad \forall f_n \in L^2(\nu_n), \]
which takes an arbitrary discrete function $f_n \in L^2(\nu_n)$ and returns a function $P^* f_n \in L^2(\nu)$ defined by
\[ P^* f_n(x) = f_n \circ T_n(x) , \quad x \in \M .\]
We also introduce a smoothing operator 
\[ \Sop_{\veps,n,0} f(x)  :=   \int_{\M}  \frac{1}{(\veps - 2 \delta_n)^m}\psi \left(\frac{d_\M(x,y)}{\veps - 2 \delta_n} \right)
  f(y) dy, \quad x \in \M , \quad f \in L^2(\nu), \]
which is a convolution operator with radial kernel 
\[   \psi (t) := \frac{1}{\alpha_\eta} \int_{t}^\infty \eta(s)s ds , \]
where we recall $\alpha_\eta:= \int_{0}^{\infty}\eta(r)r^{m+1}dr$ was defined in \eqref{alphaeta}, and where $d_\M(x,y)$ is the geodesic distance between points $x,y \in \M$.
We normalize the operator $\Sop_{\veps,n,0}$ to produce 
\[ \Sop_{\veps,n} f := \frac{1}{ \Sop_{\veps,n,0} \mathds{1}}  \Sop_{\veps,n,0} f, \quad f \in L^2(\nu), \]
where in the above $\Sop_{\veps,n,0}\mathds{1}$ denotes the application of $\Sop_{\veps,n,0}$ to the function that is identically equal to one; the normalization is introduced so as to ensure that $\Sop_{\veps,n}$ leaves constant functions unchanged. Finally, $I$ is  defined as the composition of $P^*$ with $\Sop_{\veps,n}$. Namely,
\[ I u_n := \Sop_{\veps,n} \circ P^* u_n , \quad u_n \in L^2(\nu_n)\,.\]
Note that for any $f_n\in L^2(\nu)$ it is guaranteed that $If_n\in H^1(\M, \rho)$, given that $\eta$ was assumed Lipschitz (and so in particular $\psi$ is $C^1$). \nc

\begin{lemma}[Discretization and interpolation errors]
\label{lem:Dirichlet}
Under the same assumptions on $\M$ and $\rho$ as in Theorem \ref{TheoremTranps} it follows that:
\begin{enumerate}
\item For every $f \in H^1(\M, \rho)$,
  \[   \left \lvert    \lVert Pf \rVert^2_{L^2(\nu_n)}- \lVert f \rVert^2_{L^2(\nu)}     \right \rvert \leq C_1 \left[ \delta_n \lVert f \rVert_{L^2(\nu)} D(f)^{1/2} +
      \left(\veps + \frac{\delta_n}{\veps}  \right) \lVert f \rVert^2_{L^2(\nu)}
    \right]. \]
\item  For every $f \in H^1(\M, \rho)$,
\[  b_n(Pf) \leq  \left[1+  C_2 \left(\veps + \frac{\delta_n}{\veps}\right)  \right] D(f).  \nc\]
\item  For every $f_n \in L^2(\nu_n)$ we have,
  \[  \left \lvert   \lVert  I f_n\rVert^2_{L^2(\nu)} - \lVert  f_n  \rVert^2_{L^2(\nu_n)}    \right \rvert  \leq C_3 \left[ \veps \lVert f_n \rVert_{L^2(\nu_n)} b_n(f_n)^{1/2} + \left(\veps  +  \frac{\delta_n}{\veps}\right) \lVert f_n \rVert^2_{L^2(\nu_n)}  \right] . \] 
\item For every $f_n \in L^2(\nu_n)$ such that $If_n\in H^1(\M, \rho)$,
\[ D(If_n) \leq   \left[1+C_4\left(\veps + \frac{\delta_n}{\veps}\right)  \right]  b_n(f_n). \]
\end{enumerate}
Here the constants  $C_1, C_2, C_3, C_4 >0$ only depend on lower and upper bounds for $\rho$ and its Lipschitz constant, and on $\M$ and $\eta$. The term $\delta_n$ was defined in Proposition \ref{T-n-definition}.
\end{lemma}

Before we prove Lemma \ref{lem:Dirichlet} let us introduce an intermediate, non-local continuum Dirichlet energy
\[E_r(f) := \frac{2}{\alpha_\eta} \int_{\M} \int_\M  \eta\left( \frac{|x-y|}{r} \right) |  f(x) - f(y)  |^2  \sqrt{\rho(x)} \sqrt{\rho(y)} dx dy, \quad f \in L^2(\nu), \]
where $r>0$ is a length scale to be chosen later on.

\begin{remark}
\label{rem:DEr}
Observe that the Dirichlet energy $D$ and the non-local Dirichlet energy $E_r$ can be written as
\[  D(f) = \int_{\M} \lvert \nabla f \rvert^2 \tilde{\rho}^2 dx, \quad f \in H^1(\M, \rho), \]
\[E_r(f) :=  \frac{2}{\alpha_\eta} \int_{\M} \int_\M  \eta\left( \frac{|x-y|}{r} \right) |  f(x) - f(y)  |^2  \tilde \rho(x) \tilde \rho(y) dx dy, \quad f \in L^2(\nu), \]
where $\tilde{\rho}= \sqrt{\rho}$. We notice that $\tilde{\rho}\in C^1(\M)$ since it satisfies the same regularity
condition as $\rho$ given that $\rho$ is bounded away from zero. Hence, our continuum Dirichlet energies $D$ and $E_r$
take the same form as in \cite{Moritz}. As a consequence we may use all of the bounds relating $D$ and $E_r$ that were proven there.
\end{remark}

\begin{remark}
\label{rem:bn}
The only difference between the discrete Dirichlet form $b_n$ in \eqref{def:bn} and the one introduced in \cite{Moritz} is the normalization by the terms $\sqrt{d_\veps (\x_i)/n}$ that appear in the definition of
$\W$, i.e., \cite{Moritz} uses a different normalization of
the graph Laplacian. We notice that the term  $\sqrt{d_\veps (\cdot)/n}$ uniformly approximates the density $\rho$. More precisely, given that the kernel $\eta$ is assumed to be normalized, it follows from \cite[Lemma 18]{Moritz} that
\[ \max_{i=1, \dots, n}  \left \lvert  \frac{1}{n}d_{\veps}(\x_i) - \rho(\x_i) \right \rvert \leq C \left( \veps +  \frac{\delta_n}{\veps } \right) , \]
where we recall $\delta_n$ is the $\infty$-OT distance between $\nu_n$ and $\nu$ given in
Proposition~\ref{T-n-definition}, and the constant $C>0$ depends only on $\rho$ and geometric quantities associated to $\M$. We highlight that the above is a non-optimal estimate on the error of approximation of a kernel density estimator, but has the advantage of only depending on the $\infty$-OT distance between empirical and ground-truth measures. 
\end{remark}

\begin{remark}
\label{rem:PandI}
The maps  $P, P^*, \Sop_r$ and $I$ are defined identically  to the way they were defined in \cite{Moritz} and
hence we may use all properties proved there. 
\end{remark}

\begin{proof}[Proof of Lemma \ref{lem:Dirichlet}]
  Throughout the proof we use $C$ to denote a finite
  positive constant that depends only on 
  the lower and upper bounds on $\rho$, the Lipschitz constant of $\rho$, and geometric
  quantities of $\M$. This constant may change value from one instance to the next, unless its dependence is explicitly stated.
  
The first statement follows directly from \cite[Lemma~13(i)]{Moritz} after noticing that $\rho \leq C \sqrt{\rho}$, given that $\rho$ is bounded away from zero. The second statement follows almost exactly as in \cite[Lemma~13(ii)]{Moritz}, after making some small modifications that we now describe. First, by Remark \ref{rem:bn} and the assumptions on $\rho$ it follows that for all $i, j \in \{1, \dots, n \}$ and every $x\in T_n^{-1}(\{ \x_i\})$ and $y \in T_n^{-1}(\{\x_j\})$ we have
\[ \frac{n }{\sqrt{d_\veps(\x_i) d_\veps(\x_j) } } \leq (1+ C(\veps + \delta_n/\veps ))\frac{1}{\sqrt{\rho(x) \rho(y)}}\,.\]
Using the definition of $Pf$ and Proposition~\ref{T-n-definition} (1) it follows that 
% Following the proof of \cite[Lemma 13 (ii)]{Moritz} we may then deduce that 
\begin{align*}
b_n(Pf)&= \frac{2}{\alpha_\eta n\veps^{m+2}}\sum_{i=1}^n \sum_{j=1}^n   \frac{\eta\left( \frac{|\x_i-\x_j|}{\veps} \right)}{\sqrt{d_\veps(\x_i)} \sqrt{d_\veps(\x_j) }} 
\\& \hspace{2.75cm} \cdot  \left \lvert n^2\int_{T_n^{-1}(\{ \x_i\})}\int_{T_n^{-1}(\{\x_j \})}(f(x) - f(y))  \rho(x) \rho(y)dxdy \right \rvert^2  \notag
\\ &\leq \frac{2n}{\alpha_\eta \veps^{m+2}} \sum_{i=1}^n\sum_{j=1}^n \frac{\eta\left( \frac{|\x_i-\x_j|}{\veps} \right)}{\sqrt{d_\veps(\x_i)} \sqrt{d_\veps(\x_j) }}\\
&\hspace{2.75cm} \cdot \int_{T_n^{-1}(\{ \x_i\})}\int_{T_n^{-1}(\{\x_j \})}(f(x) - f(y) )^2 \rho(x) \rho(y)dxdy  \notag
\\& \leq \frac{ 2(1+ C (\veps + \delta_n /\veps))}{\alpha_\eta \veps^{m+2}}\sum_{i=1}^n \sum_{j=1}^n \eta\left(\frac{|\x_i-\x_j|}{\veps} \right)\\
&\hspace{2.75cm}\int_{T_n^{-1}(\{ \x_i\})}\int_{T_n^{-1}(\{\x_j \})}(f(x) - f(y) )^2 \sqrt{\rho}(x) \sqrt{\rho}(y)dxdy\,,
\end{align*}
where the first inequality follows by Jensen's inequality. Using the above bound together with Remark~\ref{rem:DEr} and proceeding identically to the proof
of \cite[Lemma~13(ii)]{Moritz} we can bound $b_n$ in terms of $E_r$. More precisely,
\begin{equation*}
  b_n(Pf) \leq \frac{ 1+ C( \veps + \delta_n /\veps) }{\veps^{m+2}}
  \left( E_{\veps' + 2 \delta_n}(f) + \frac{\tilde C \delta_n}{\veps} E_{2(\veps' + 2 \delta_n)}(f)   \right),
\end{equation*}
where $\veps' = (1 + C' \veps^2)\veps$, $C'$ is a constant intrinsic to
the manifold $\M$, and $\tilde C>0$ is a constant
depending only on $\eta$. At this point we can follow the rest of the proof
of \cite[Lemma~13(ii)]{Moritz} to obtain the desired result after bounding $E_r$
in terms of $D$:
\begin{align*}
\frac{1}{\veps^{m+2}} E_{\veps' + 2 \delta_n}(f)
&\le \left(1+C \left(\veps+\frac{\delta_n}{\veps}\right)\right) D(f)\,,\\
\frac{1}{ \veps^{m+2}}\frac{ \delta_n}{\veps} E_{2(\veps' + 2 \delta_n)}(f)
&\le C \frac{ \delta_n}{\veps}  \left( 1+ \veps+\frac{\delta_n}{\veps}\right) D(f).
\end{align*}

The third statement follows directly from \cite[Lemma 14 (i)]{Moritz} after noticing that thanks to Remark \ref{rem:bn} we have $\sqrt{d_\veps(\x_i) d_\veps(\x_j) }/n \leq C   $, where $C$ depends only on the upper bound for $\rho$. Indeed, this inequality allows us to upper bound the discrete Dirichlet energy introduced in \cite{Moritz} with a constant multiple of our discrete Dirichlet energy $b_n$. 
 \nc
 
For the fourth and final statement we proceed as in  \cite[Lemma 14(ii)]{Moritz},
but making a small modification. By Remark \ref{rem:bn} we
conclude that for all $i, j \in \{1, \dots, n \}$ and every $x\in T_n^{-1}(\{ \x_i\})$ and $y \in T_n^{-1}(\{\x_j\})$ we have
\[  \frac{1}{\sqrt{\rho(x) \rho(y)}} \leq \frac{n  (1+ C(\veps + \delta_n/\veps))}{\sqrt{d_\veps(\x_i) d_\veps(\x_j) } }. \]
The above inequality is used to replace the degree term $d_\veps/n$ with the density $\rho$, allowing us to reverse the bound in Lemma~\ref{lem:Dirichlet}(2) \nc. Having in mind the inequality, the proof of our statement is exactly as in \cite{Moritz}. Indeed,  we can follow \cite[Lemma~14(i)]{Moritz} and show that
\begin{equation}\label{bound-E-by-b-n}
  E_{\veps - 2 \delta_n}(P^\ast f_n)  \le \nc\left( 1 +  C \left( \veps +\frac{\delta_n}{\veps}\right) \right) \veps^{m+2}b_n(f_n). 
\end{equation}
In turn, \cite[Lemma~14(ii)]{Moritz} (recall remark \ref{rem:DEr}) gives
\begin{equation*}
  D(If_n) \le \left[1 + C \left(\veps + \frac{\delta_n}{\veps} \right) \right]
  \frac{1}{\veps^{m+2}} E_{\veps - 2 \delta_n} (P^\ast f_n),
\end{equation*}
for all $f_n \in L^2(\nu_n)$. Combining \eqref{bound-E-by-b-n} with the above inequality we deduce the desired result.

\end{proof}

With Lemma~\ref{lem:Dirichlet} at hand, we can obtain a precise relationship between eigenvalues of $\Delta_n$ and $\Delta_\rho$.

\begin{lemma}[Convergence rate for eigenvalues ]
\label{lem:eigen}
Let $\lambda_i$ be the $i$-th eigenvalue of $\Delta_\rho$ and let $\lambda_{n,i}$ be the $i$-th eigenvalue of $\Delta_n$.  Let $\beta>1$. Then
there exist constants $C, C_\beta >0$  such that for
sufficiently large $n$, with probability at least $1- C_\beta n^{-\beta}$, we have
\[ \lvert \lambda_{n,i} - \lambda_i \lvert\leq
  C \left( \veps + \frac{\delta_n}{\veps} +  \delta_n\sqrt{\lambda_i}\right)  \lambda_{i}, \quad i=1, \dots, N,  \]
where $C>0$ depends only on $\M$, $\eta$ and $\rho$ and $C_\beta >0$ depends only on $\beta$.
\end{lemma}

\begin{proof}
Using Lemma \ref{lem:Dirichlet}, we can follow exactly the proof of \cite[Theorem 4]{Moritz}. 
\end{proof}

Now, we are ready to prove Theorem \ref{TheoremTranps}.

\begin{proof}[Proof of Theorem \ref{TheoremTranps}]
  Throughout this proof we use $C$ to denote a finite positive constant that
  depends only on $\M$, $N$, $\rho^\pm$, $C_{Lip}$ and $\eta$.
  This constant may change from one instance to the next.
  
  Let $u_{n,1}, \dots, u_{n,N}$ be the first $N$ eigenfunctions of $\Delta_n$ with unit norm. We can assume they form an orthonormal basis with respect to $\langle\cdot,\cdot\rangle_{\nu_n}$, with corresponding eigenvalues 
   \[ \lambda_{n,1} \leq \dots \leq \lambda_{n, N}. \]
We now put \cite[Lemma 7.3]{BIK} together with Lemmata \ref{lem:Dirichlet} and \ref{lem:eigen}, as well as Remark \ref{rem:PandI} to conclude that for every $j=1, \dots, N$ we have
\begin{equation}\label{I-Proj-dist}
  \lVert  Iu_{n,j} - \PN( I u_{n,j}  ) \rVert^2_{L^2(\nu)} \leq \frac{C_{\M, N} \lambda_{N}}{\lambda_{N+1}-\lambda_{N}}\left(  \veps  + \frac{\delta_n}{\veps}   \right) =: \gamma_0^2,
\end{equation}
where 
% we recall $\delta_n$ is the $\infty$-OT cost of the maps $T_n$ introduced at the beginning of this section. In the above 
$\PN$ denotes the projection onto $U$, the span of the first $N$ eigenfunctions of $\Delta_\rho$
and $C_{\M, N}>0$ is a constant depending on $\M$ and $N$ only.

We need to show that the functions $\PN (Iu_{n,1}), \dots, \PN(Iu_{n,N})$ can
be modified slightly to form an orthonormal basis for $U$. First, by \eqref{I-Proj-dist} we have
\begin{align*}
  \| I u_{n,j} \|_{L^2(\nu)} - \gamma_0 \le \| \PN (I u_{n,j}) \|_{L^2(\nu)} \le
   \| I u_{n,j} \|_{L^2(\nu)} + \gamma_0\,.
\end{align*}
Next, we find a bound on $\| I u_{n,j} \|_{L^2(\nu)}$ using Lemma~\ref{lem:Dirichlet}(3). To control the $b_n(u_{n,j})$ term on the right hand side of Lemma~\ref{lem:Dirichlet}(3) we make use of the bound on the eigenvalues provided in Lemma~\ref{lem:eigen},
\begin{align}
\label{bnest}
\begin{split}
 b_n(u_{n,j})&= \langle u_{n,j}\,,\,\Delta_n u_{n,j}\rangle_{L^2(\nu_n)}
 =  \lambda_{n,j}\\
 &\le C \left( 1+\veps + \frac{\delta_n}{\veps} +  \delta_n\sqrt{\lambda_N}
    \right)  \lambda_{N}
  \\& \leq C\lambda_N  \,,
 \end{split}
\end{align}
and so, since the $u_{n,j}$ are normalized, we obtain
\begin{equation}\label{Iubound}
 \left|\| I u_{n,j} \|_{L^2(\nu)}^2-1\right|
 \le C \left(\veps\sqrt{\lambda_N}    +  \veps + \frac{\delta_n}{\veps} \right).
\end{equation}
Combining with the estimate for $  \lVert \PN I u_{n,j} \rVert_{L^2(\nu)}$ yields (for $\veps$ and $\delta_n/\veps$ small enough and $n$ large enough so that the right-hand side in the last estimate is less or equal to 1) 
\begin{equation}\label{PN-close-to-Id}
  1-  \gamma_1 \leq \lVert \PN(I u_{n,j})  \rVert_{L^2(\nu)} \leq  1 + \gamma_1 \quad \forall j=1, \dots, N,
\end{equation}
where
\begin{align*}
  \gamma_1:= C \left[  \veps\sqrt{\lambda_N}    +  \veps + \frac{\delta_n}{\veps}  \right]^{1/2} + \gamma_0\,.
\end{align*}

Lemma~\ref{lem:Dirichlet}(3) also allows us to bound the difference of inner products. For $i \not = j$ we have 
\[ \langle I u_{n,j}, I u_{n,i} \rangle_{L^2(\nu)} =  \frac{1}{2}\left(  \lVert I u_{n,j} \rVert^2_{L^2(\nu)} + \lVert I u_{n,i} \rVert^2_{L^2(\nu)} - \lVert I u_{n,j}- Iu_{n,i} \rVert^2_{L^2(\nu)}  \right), \]
\[ 0  =\langle u_{n,j},  u_{n,i} \rangle_{L^2(\nu_n)} =  \frac{1}{2}\left(  \lVert  u_{n,j} \rVert^2_{L^2(\nu_n)} + \lVert  u_{n,i} \rVert^2_{L^2(\nu_n)} - \lVert  u_{n,j}- u_{n,i} \rVert^2_{L^2(\nu_n)}  \right). \]
Subtracting the above identities and using Lemma~\ref{lem:Dirichlet}(3) we obtain
\begin{align*}
  \rvert \langle Iu_{n,j} , Iu _{n,i} \rangle_{L^2(\nu)} \lvert  \leq  C \left(\veps\sqrt{\lambda_N}    +  \veps + \frac{\delta_n}{\veps} \right).
\end{align*}
%Note that the eigenfunctions are normalized, and $b_n( u_{n,i} + u_{n,j})\le 2\left(b_n( u_{n,i})+b_n( u_{n,j})\right)$. We conclude using \eqref{bnest}.
%\blue
%\begin{align*}
% & |\langle I u_{n,i}, I u_{n,j} \rangle_{\rho}| 
%  \le\| I ( u_{n,i} \pm u_{n,j}) \|_{L^2(\nu)}^2 \\
%&\quad \le C \bigg[ \veps \left( b_n( u_{n,i})^{1/2} + b_n(u_{n,j})^{1/2} \right)
% +\left( 1+ \veps + \delta_n / \veps + \delta_n +\mathcal{O}(\veps^2)\right) \bigg]\,\\
% &\quad \le C \bigg[ \veps \sqrt{\frac{2}{n}}
% \left( 1+\veps + \frac{\delta_n}{\veps} +  \delta_n\sqrt{\lambda_N}\right)^{1/2}  \lambda_{N}^{1/2}
% +1+ \veps + \delta_n / \veps + \delta_n +\mathcal{O}(\veps^2)\bigg]:= C+\gamma_2\,.
%\end{align*}
%\nc
From this, and the fact that both $ \lVert I{u_{n,j}} - \Pi_N Iu_{n,j}  \rVert_{L^2(\nu)}$ and  $ \lVert I u_{n,i} - \Pi_N Iu_{n,i}  \rVert_{L^2(\nu)} $ are smaller than $\gamma_0$, we deduce that 
\begin{align}
\begin{split}
| \langle \Pi_N I u_{n,i} , \Pi_N I u_{n,j}  \rangle_{L^2(\nu)} | & \leq |\langle I u_{n,i} , I u_{n,j}  \rangle_{L^2(\nu)}|  +  \lvert  \langle I u_{n,j}, I u_{n,i} - \Pi_N I u_{n,i} \rangle_{L^2(\nu)}  \rvert
\\& + \lvert  \langle \Pi_N I u_{n,i}, I u_{n,j} - \Pi_N I u_{n,j} \rangle_{L^2(\nu)}  \rvert
\\ & \leq   C \left(\veps\sqrt{\lambda_N}    +  \veps + \frac{\delta_n}{\veps} \right) +  C\gamma_0=: \gamma_2 , \quad \forall i \not = j,
\end{split}
\label{innerprodproj}
\end{align}
where the second inequality follows from Cauchy-Schwartz, the fact that $\|\Pi_N  u\|_{L^2(\nu)} \le \|u\|_{L^2(\nu)} $ for any $u\in L^2(\nu)$, and from $\|I u_{n,i}\|_{L^2(\nu)}\le C$ by \eqref{Iubound}.

With \eqref{PN-close-to-Id}, \eqref{innerprodproj} and Assumptions \ref{assumption-on-rho-contrast} (guaranteeing the smallness of $\gamma_0$, $\gamma_1$ and $\gamma_2$) we can use Lemma~\ref{BH-finite-dim-orht-basis} and Remark \ref{App:linearIndep}, to deduce the existence of an orthonormal system $g_1, \dots, g_N$ for $U$ satisfying: 
\[ \lVert\PN I u_{n,j} - g_j\rVert_{L^2(\nu)}   \leq
\sqrt{N} \left( \frac{1}{\sqrt{1 - N \gamma_2}} -1 \right), \forall j=1, \dots, N . \]

Using \eqref{I-Proj-dist} and expanding in $N \gamma_2$ we deduce 
 \begin{align*}
   \lVert I u_{n,j} - g_j\rVert_{L^2(\nu)}^2 
 &\leq \left(\lVert I u_{n,j} - \PN I u_{n,j}\rVert_{L^2(\nu)}  +\lVert \PN I u_{n,j} - g_j\rVert_{L^2(\nu)}\right)^2\\
 &\leq
   \left( \gamma_0 + \sqrt{N} \left( \frac{1}{\sqrt{1 - N \gamma_2}} -1 \right) \right)^2 \\
   &\le C \left( \gamma_0 +\frac{N^{3/2}}{2}\gamma_2+\frac{3N^{5/2}}{8}\gamma_2^2 \right)^2 
   =: \gamma_3, \quad \forall j=1, \dots,N.\nc
 \end{align*}
Now,
\begin{align*}
\lVert Iu_{n,j} - u_{n,j}\circ T_n \rVert_{L^2(\nu)}^2
&= \lVert \Sop_{\veps - 2 \delta_n} P^* u_{n,j} - P^*u_{n,j} \rVert_{L^2(\nu)}^2
\\& \leq\frac{ C \veps^2}{(\veps-2 \delta_n)^{m+2}}E_{\veps-2\delta_n}(P^* u_{n,j}) 
\\& \leq  C \veps^2  b_n(u_{n,j}) 
  \\
  & \leq  C \veps^2  \lambda_{N} =: \gamma_4,
\end{align*}
\nc
where we have used \cite[Lemma 8]{Moritz}, assuming $\delta_n < \veps/(m+ 5)$  to obtain the first inequality and where the second and third inequalities follow from \eqref{bound-E-by-b-n} and \eqref{bnest} respectively.
Using triangle inequality and the above estimates we obtain
\[
 \lVert u_{n,j} \circ T_n -g_j\rVert_{L^2(\nu)}^2 \le 2 ( \gamma_3 +  \gamma_4) .
\]
%Substituting the definition of $\gamma_3$ and $\gamma_4$, rearranging
%and discarding terms of higher order than $\veps^2, (\delta_n/\veps)^2$ and $\delta_n^2$ 
%we
%obtain
%\begin{align*}
%  \lVert u_{n,j} \circ T_n -g_j\rVert_{L^2(\nu)}^2 &\le C\Bigg[
%   \left( \left( \frac{\lambda_N}{\lambda_{N+1} - \lambda_N} \right)
%      \left( \veps + \delta_n/\veps \right) \right)^2
%  \\ & \quad  + N \left( \left( 1 - N \left(  \veps + \delta_n +
%       \frac{\lambda_N}{\lambda_{N+1} - \lambda_N}\right) ( \veps + \delta_n /\veps + \delta_n)
%         \right)^{-1} - 1 \right)
%  \\ & \quad +
%       \veps^2 ( 1 + \veps + \delta_n/\veps + \delta_n \sqrt{\lambda_N} ) \lambda_N\Bigg]
%  \\ & \le c_\M \Bigg[ \left(\lambda_N + \left( \frac{\lambda_N}{ \lambda_{N+1} - \lambda_{N}} \right)^2 \right)
%       \left( \veps + \delta_n/\veps + \delta_n \sqrt{\lambda_N} \right)^2
%       \\ & \quad + N \left( \left( 1 - N \left(
%       \frac{\lambda_N^{1/2} + \lambda_N}{\lambda_{N+1} - \lambda_N}\right) ( \veps + \delta_n /\veps + \delta_n)
%         \right)^{-1} - 1 \right) \Bigg].
%\end{align*}
%\red
%The lowest order term in $\gamma_3$ is $\gamma_0^2$, which is of first order. $\gamma_4$ is of order $\veps^2$. Hence, up to second oder in $\veps, \delta_n/\veps$ and $\delta_n$, we have 

Plugging in the definitions of $\gamma_3$ and $\gamma_4$ and using Assumptions \ref{assumption-on-rho-contrast}, we may collect the highest order terms and deduce that for all $j=1, \dots, N$ we have
      \begin{align*}
       \lVert u_{n,j} \circ T_n -g_j\rVert_{L^2(\nu)}^2 &\le 
        C(1+N^{3/2}+N^3) C_{\M,N} \left( \frac{\lambda_N}{\lambda_{N+1} - \lambda_N} \right)
      \left( \veps + \delta_n/\veps \right)\,.
      \end{align*}
      where $C_{\M,N}$ is the constant in  \eqref{I-Proj-dist}.
 Proposition~\ref{T-n-definition}(2) implies that
      \begin{align*}
       \lVert u_{n,j} \circ T_n -g_j\rVert_{L^2(\nu)}^2 &\le c_{\M} \left( \frac{\lambda_N}{\lambda_{N+1} - \lambda_N} \right)
      \left( \veps + \frac{\log(n)^{p_m}}{\veps n^{1/m}} \right)\,,
      \end{align*}
      where $c_\M$ is a constant proportional to $C_{\M, N}$ and $C$ given in Proposition~\ref{T-n-definition}(2). This concludes the proof of the theorem.
\end{proof}
{}

Our goal is now to replace the terms $\lambda_N$ and $\lambda_{N+1}$ in \eqref{thm:erroreigen} with quantities that only depend on the parameters of the mixture model. A lower bound for $\lambda_{N+1}$ was already obtained in Proposition \ref{auxLemmaSigmaB}, and now we focus on obtaining an upper bound for $\lambda_{N}$. 

\begin{proposition}[Upper bound for $\lambda_N$]
\label{UppLambdaN}
Suppose that $\mathcal{S}^{1/2}N < 1$ and let $\lambda_N$ be the $N$-th eigenvalue of $\Delta_\rho$. Then,
\[ \lambda_N \leq \frac{N \mathcal{C}}{1- N \S^{1/2}}\,. \]
\end{proposition}

\begin{proof}

To get an upper bound for $\lambda_{N}$ it is enough to use the min-max formula \cite[Theorem~8.4.2]{Buttazzo} for $\lambda_N$. In particular, since $Q:= \Span\{ q_1, \dots, q_N\}$ is $N$-dimensional it follows that 
\[\lambda_{N} \leq \max_{u \in Q}  \frac{ \langle \Delta_\rho u ,  u \rangle_\rho}{\langle u , u \rangle_\rho}. \]

Let $u \in Q$ for which $\langle u, u \rangle_\rho =1$. Then,
\[ u= \sum_{i=1}^N a_k q_k \]
for some scalars $a_k$ satisfying
\[ 1= \sum_{k=1}^N a_k^2 \lVert q_k \rVert_\rho^2 + \sum_{k=1}^N \sum_{j\not = k} a_k a_j \langle q_k, q_j \rangle_\rho  = \sum_{k=1}^N a_k^2 w_k + \sum_{k=1}^N \sum_{j\not = k} a_k a_j \langle q_k, q_j \rangle_\rho. \]
Now, 
\[ \left | \sum_{k=1}^N \sum_{j\not = k} a_k a_j \langle q_k, q_j \rangle_\rho \right| \leq \sum_{k=1}^N \sum_{j \not = k} |a_k| |a_j| \sqrt{w_k}\sqrt{w_j} \mathcal{S}^{1/2}  \leq N \S^{1/2} \sum_{k=1}^N a_k^2 w_k. \]
Thus, 
\[ \sum_{k=1}^N a_k^2 w_k \leq   \frac{1}{1- N\mathcal{S}^{1/2}}. \]
In addition, 
\[ \langle \Delta_\rho u , u \rangle_\rho = \sum_{k=1}^N\sum_{j=1}^N a_k a_j \int_{\M} \nabla q_k \cdot \nabla q_j  \rho dx \,.\]
% = \frac{1}{4} \sum_{k=1}^N\sum_{j=1}^N a_k a_j \int_{\M} \nabla q_k \cdot \nabla q_j  \rho dx  \]
Recall that $\lVert \nabla q_k\rVert_\rho^2=w_k \C_k$. Using H\"older's inequality, we obtain
\begin{align*}
  \langle \Delta_\rho u , u \rangle_\rho
&\le \sum_{k=1}^N \sum_{j=1}^N |a_k a_j| \left( \int_\M |\nabla q_k|^2 \rho dx  \right)^{1/2} \left( \int_\M |\nabla q_j|^2 \rho dx  \right)^{1/2} \\
&\leq   \mathcal{C} N \sum_{k=1}^N a_k^2 w_k
\leq \frac{N \mathcal{C}}{1- N \S^{1/2}}\,.
\end{align*}
We  conclude that 
\[ \lambda_N \leq \frac{N\mathcal{C}}{1-N\S^{1/2}}.\]
\end{proof}

\begin{remark}
In the previous result we used the trace of the operator $\Delta_\rho$ restricted to $Q$ to bound $\Ind_N$. This is not necessarily an optimal bound, but we remark that in the case when $\mathcal{C}/\Ind$ is small enough (in particular smaller than $1/N$), one can use this estimate to get a meaningful lower bound for the spectral gap $\lambda_{N+1}- \lambda_N$. More precisely, we obtain
\[ \lambda_{N+1}- \lambda_{N}\geq \left( \sqrt{\Ind(1 -N\S) } - \frac{\sqrt{\C N \S}}{(1-\S)} \right)^2  -\frac{N \C}{1- N \S^{1/2}}  .\]
While in general we think of $N$ as a relatively small number (representing the number of meaningful components in a data set) we would like to remark that $N$ in the estimates 
\[ \lambda_N \leq \frac{N\C}{1- N \S^{1/2}},\]
 can be replaced with $N_{eff}$ where $N_{eff}$ is the effective number of components any given component intersects. That is,
\[ N_{eff} = \max_{k=1, \dots, N}  \# \{ i\neq k \text{ s.t } \nu( supp(\rho_k)\cap supp(\rho_i))>0  \}. \]
 Estimates with better dependence on $N$ may be obtained in terms of a quantity analogue to $\C$ of the form:
\[ \max_{k =1, \dots, N} \int_{\M} | \nabla \log(\rho_k/\rho) |^3 \rho_k dx.\]

\end{remark}

Combining the lower bound for $\lambda_{N+1}$ from Proposition~\ref{auxLemmaSigmaB} and the upper bound for $\lambda_{N}$ we immediately deduce the following.

\begin{corollary}\label{cor:Transp}In Theorem \ref{TheoremTranps}, inequality \eqref{thm:erroreigen} can be replaced with:
\begin{equation}
  \begin{split}
    \int_{\M} | g_j(x)  - u_{n,j}\circ T_n(x)   |^2 d\nu(x) = \lVert g_j - u_{n,j}\circ T_n \rVert_{L^2(\nu)}^2& \le
    \phi(\S, \C, \Theta, N, \veps, n, m),
     \end{split}
\label{thm:erroreigenCor}
\end{equation}
for all $j=1, \dots, N$, where we recall
\begin{align}\label{def:phi}
  \phi(\S, \C, \Theta, N, \veps, n, m)& := c_\M
   \left(\frac{N\C}{1- N \S^{1/2}} \right)
  \left( \veps +  \frac{\log(n)^{p_m}}{\veps n^{1/m}}\right) \notag\\
       & \cdot  \left( \left( \sqrt{\Theta (1 - N \S)} - \frac{\sqrt{\C N \S}}{1- \S}\right)^2 - \frac{N \C}{1- N \S^{1/2}} \right)^{-1} ,
\end{align}
%\begin{align}
%\phi(\S, \C, \Theta, N, \veps, n, m)& = C_\M
%\Bigg[ \left(\frac{N\C}{1- N \S^{1/2}} \right)
%\left( \veps +  \frac{\log(n)^{p_m}}{\veps n^{1/m}}
%+ \frac{\log(n)^{p_m}}{ n^{1/m}} \left(\frac{N\C}{1- N \S^{1/2}} \right)^{1/2} \right)^2 \\
%& + N \left( \left( 1 - N
%\left(\frac{N\C}{1- N \S^{1/2}} \right)^{1/2} \left( \veps +
%\frac{\log(n)^{p_m}}{\veps n^{1/m}} + \frac{\log(n)^{p_m}}{ n^{1/m}}\right)
%\right)^{-1} - 1 \right) \Bigg],
%\end{align}
and 
$c_\M$ is a constant depending on $\M$, $\rho^\pm$, $C_{Lip}$, $\eta$ and $N$.
\end{corollary}

%%%%%%%%%%%%%%%%%%%%%%%%%%%%%%%%%%%%%%%%%%%%%%%%%%%%%%%%%%%%%%%%%%%%%%
\subsection{Proof of Theorem \ref{PropMuMun} }
 \label{sec:proof2}

 \begin{proof}[Proof of Theorem \ref{PropMuMun}]
Let $\beta>1$. From Corollary~\ref{cor:Transp}, we know that with probability greater than $ 1- C_\beta n^{-\beta}$, there exist a transportation map $T_n : \M \rightarrow \{\x_1, \dots, \x_n \}$ pushing forward $\nu$ into $\nu_n$ and an orthonormal set of functions $g_1, \dots, g_n$ in $U$ (the space generated by the first $N$ eigenfunctions of $\Delta_\rho$) satisfying
\begin{enumerate}
\item  $\sup_{x \in \M} d_\M(x, T_n(x)) \leq c_\M \frac{\log(n)^{p_m}}{n^{1/m}}$.
\item $\int_{\M} |g_i(x) - u_{i,n}(T_n(x))|^2 d \nu(x) \leq
  \phi(\S, \C, \Theta, N, \veps, n, m),$
\end{enumerate}
where the function $\phi$ is defined in \eqref{def:phi}
and used throughout this proof for convenience of notation.
Let $G(x) := ( g_1(x), \dots, g_N(x) )$ for $x \in \M$.  We deduce that  
\begin{align*}
&\int_{\M} \left| F_n\circ T_n (x) - G(x)   \right|^2  d \nu(x) 
= \sum_{i=1}^N \int_{\M}  | g_i(x) -u_{n,i}(T_n(x)) |^2 d \nu(x) 
\\ & \quad \leq N  \phi(\S, \C, \Theta, N, \veps, n, m).
\end{align*}
 We claim that the integral on the left hand side of the previous expression can be written as:
 \begin{equation*}
 \int_{\M} \left| F_n\circ T_n (x) -G(x)   \right|^2  d\nu(x)  =    \int_{\R^N  \times \R^N}   | x- y|^2 d \pi_n(x,y),   
 \label{aux11}
 \end{equation*}
 for a transportation plan $\pi_n \in \mathcal P(\R^N\times \R^N)$
 between the measures $G_\sharp \nu$ and $F_{n \sharp} \nu_n$. To see this,  let $\tilde{\pi}_n \in  \mathcal P(\M\times \M)$ be given by
 \[  \tilde{\pi}_n := (Id \times T_n) _{\sharp } \nu \,, \] 
 and let
 \[G\times F_n : \M \times \M \rightarrow \R^N \times \R^N, \]
 be given by
 \[G\times F_n : (x,y) \longmapsto  (G(x), F_n(y)).  \]
 Let $\pi_n  := (G \times F_n  )_{\sharp } \tilde{\pi}_n$ (i.e. the push-forward of $\tilde \pi_n$ by the map $G  \times F_n$ ). It is straightforward to check that $\pi_n$ is a
 transportation plan between $G_{\sharp}\nu$ and $ F_{n \sharp} \nu_n$, and moreover, 
 \begin{align*}
 \begin{split}
 \int_{\R^N \times \R^N} | x - y |^2 d \pi_n(x,y) &= \int_{\R^N \times \R^N} |x- y |^2  d (G \times F_n )_{\sharp } \tilde{\pi}_n(x,y) 
 \\& =  \int_{\M \times \M} | G(x)- F_n(y) |^2   d\tilde{\pi}_n(x,y)  
 \\& = \int_{\M \times \M} | G(x) - F_n(y) |^2 d (Id \times T_n)_{\sharp}\nu (x,y)
 \\& = \int_{\M} |G(x) - F_n\circ T_n (x) |^2 d \nu(x).
 \end{split}
 \end{align*}
 Therefore,
 \begin{align}\label{estGF_n}
 (W_2(G_{\sharp} \nu, F_{n \sharp} \nu_n))^2 &\leq  \int_{\M} \left| F_n\circ T_n (x) - G(x)   \right|^2  d\nu(x)\notag \\
&\leq N   \phi(\S, \C, \Theta, N, \veps, n, m).
 \end{align}

Since $g_1, \dots, g_N$ is an orthonormal basis for $U$, we conclude that there exists an orthogonal matrix $R$ such that for every $x \in \M$ we have
\[ G(x) = R F(x),  \]
where $F$ is the continuum spectral embedding as defined in \eqref{ContEmbedding}.

In order to show that the measure $\mu_n:=F_{n \sharp } \nu_n$ has an orthogonal cone structure, it is enough to show that the pushforward of $\nu_n$ through any orthogonal transformation of $F_n$ has an orthogonal cone structure. In particular, we will show that the measure $(R O^{-1} F_n)_{\sharp} \nu_n$ has an orthogonal cone structure,  where $R$ is defined as above, and $O$  is the orthogonal transformation chosen in the proof of Theorem~\ref{mainTheorem}:
\[ OF(x)=\tilde F (x) = \sum_{j=1}^N \frac{\PN(q_j)(x)}{\|\PN(q_j)\|_\rho} e_j\,.\]
Indeed, the measures $R^{-1}OF_{\sharp} \nu$ and $F^Q_{\sharp} \nu$ are close to each other in the Wasserstein sense: since $G=RO^{-1}\tilde F$,
\begin{align*}
 W_2(OR^{-1}F_{n \sharp} \nu_n, F^Q_\sharp \nu) 
 &\leq  W_2(OR^{-1}F_{n \sharp} \nu_n, \tilde F_\sharp \nu) + W_2(\tilde F_\sharp\nu, F^Q_\sharp \nu) \\
  &= W_2(F_{n \sharp} \nu_n, RO^{-1}\tilde F_\sharp \nu) + W_2(\tilde F_\sharp\nu, F^Q_\sharp \nu) \\
  &= W_2(F_{n \sharp} \nu_n, G_\sharp \nu) + W_2(\tilde F_\sharp\nu, F^Q_\sharp \nu) \\
&\leq \sqrt{N     \phi(\S, \C, \Theta, N, \veps, n, m)   }\\
&\qquad + \sqrt{N\left(\frac{\tau-\sqrt{\S}}{2}\right)^2
+ 4N^{3/2} \left(\frac{1}{\sqrt{1-N \tau}} -1 \right)}\,,
\end{align*}
where the last inequality follows from \eqref{estGF_n} and \eqref{estFQFtilda}. Thanks to the assumption~\ref{ineq:mainTheorem2} we can apply Proposition~\ref{prop:muQcone} and Proposition~\ref{LemmaConeWasserstein} to conclude that $OR^{-1}F_{n \sharp} \nu_n$ has an orthogonal cone structure with the parameters as stated in Theorem~\ref{PropMuMun}, and hence so does $F_{n \sharp}\nu_n$. 
 \end{proof}

%%%%%%%%%%%%%%%%%%%%%%%%%

 {\footnotesize \textbf{Acknowledgements}
   The authors would like to thank Ulrike von Luxburg
   for pointing them to the paper \cite{BinYu} which was the
   starting point of this work. 
   Franca Hoffmann was partially supported by Caltech's
von Karman postdoctoral instructorship. Bamdad Hosseini is supported in part by a postdoctoral fellowship granted by Natural Sciences and Engineering Research Council of Canada.}

%%%%%%%%%%%%%%%%%%%%%%%%%%%%

\bibliography{spectral_clustering}

\begin{thebibliography}{10}

\bibitem{anne1995note}
{\sc C.~Ann{\'e}}, {\em A note on the generalized dumbbell problem},
  Proceedings of the American Mathematical Society, 123 (1995), pp.~2595--2599.

\bibitem{Buttazzo}
{\sc H.~Attouch, G.~Buttazzo, and G.~Michaille}, {\em Variational analysis in
  {S}obolev and {BV} spaces: Applications to PDEs and optimization}, MOS-SIAM
  Series on Optimization, SIAM,Philadelphia, 2014.

\bibitem{bel_niy_LB}
{\sc M.~Belkin and P.~Niyogi}, {\em Towards a theoretical foundation for
  {L}aplacian-based manifold methods}, J. Comput. System Sci., 74 (2008),
  pp.~1289--1308.

\bibitem{bogachev1}
{\sc V.~I. Bogachev}, {\em Measure Theory}, vol.~1, Springer, New York, 2007.

\bibitem{BIK}
{\sc D.~Burago, S.~Ivanov, and Y.~Kurylev}, {\em A graph discretization of the
  {L}aplace-{B}eltrami operator}, J. Spectr. Theory, 4 (2014), pp.~675--714.

\bibitem{Cheeger}
{\sc J.~Cheeger}, {\em A lower bound for the smallest eigenvalue of the
  laplacian}, in Proceedings of the Princeton conference in honor of Professor
  S. Bochner, 1969, pp.~195--199.

\bibitem{chung1997spectral}
{\sc F.~R.~K. Chung}, {\em Spectral graph theory}, no.~92 in CBMS: Regional
  Conference Series in Mathematics, American Mathematical Society, Providence,
  1997.

\bibitem{Coifman1}
{\sc R.~R. Coifman and S.~Lafon}, {\em Diffusion maps}, Appl. Comput. Harmon.
  Anal., 21 (2006), pp.~5--30.

\bibitem{StuartSlepcevThorpeDunlop}
{\sc M.~M. {Dunlop}, D.~{Slep{\v{c}}ev}, A.~M. {Stuart}, and M.~{Thorpe}}, {\em
  {Large Data and Zero Noise Limits of Graph-Based Semi-Supervised Learning
  Algorithms}}, arXiv e-prints,  (2018), p.~arXiv:1805.09450.

\bibitem{Moritz}
{\sc N.~{Garcia Trillos}, M.~{Gerlach}, M.~{Hein}, and D.~{Slepcev}}, {\em
  {Error estimates for spectral convergence of the graph Laplacian on random
  geometric graphs towards the Laplace--Beltrami operator}}, arXiv e-prints,
  (2018), p.~arXiv:1801.10108.

\bibitem{GarciaMurray}
{\sc N.~Garc\'{i}a~Trillos and R.~Murray}, {\em A new analytical approach to
  consistency and overfitting in regularized empirical risk minimization},
  European J. Appl. Math., 28 (2017), pp.~886--921.

\bibitem{GarciaSanzAlonso}
{\sc N.~Garc\'{i}a~Trillos and D.~Sanz-Alonso}, {\em Continuum limits of
  posteriors in graph {B}ayesian inverse problems}, SIAM J. Math. Anal., 50
  (2018), pp.~4020--4040.

\bibitem{GarciaTrillosSlepcev2015}
{\sc N.~Garc{\'i}a~Trillos and D.~Slep{\v{c}}ev}, {\em Continuum limit of total
  variation on point clouds}, Archive for Rational Mechanics and Analysis,
  (2015), pp.~1--49.

\bibitem{GTSSpectralClustering}
{\sc N.~Garc\'{i}a~Trillos and D.~Slep\v{c}ev}, {\em A variational approach to
  the consistency of spectral clustering}, Appl. Comput. Harmon. Anal., 45
  (2018), pp.~239--281.

\bibitem{HeinvonLuxburgAudibert}
{\sc M.~Hein, J.-Y. Audibert, and U.~von Luxburg}, {\em From graphs to
  manifolds---weak and strong pointwise consistency of graph {L}aplacians}, in
  Learning theory, vol.~3559 of Lecture Notes in Comput. Sci., Springer,
  Berlin, 2005, pp.~470--485.

\bibitem{higham-nearness-matrix}
{\sc N.~J. Higham}, {\em Matrix nearness problems and applications}, in
  Applications of matrix theory, Oxford University Press, 1989, pp.~1--27.

\bibitem{FHBHAS}
{\sc F.~Hoffmann, B.~Hosseini, A.~Oberai, and A.~M. Stuart}, {\em Posterior
  consistency of graph-based probit in the continuum limit}, in preparation,
  (2019).

\bibitem{LingStromer}
{\sc S.~{Ling} and T.~{Strohmer}}, {\em {Certifying Global Optimality of Graph
  Cuts via Semidefinite Relaxation: A Performance Guarantee for Spectral
  Clustering}}, arXiv e-prints,  (2018), p.~arXiv:1806.11429.

\bibitem{LittleMaggioniMurphy}
{\sc A.~{Little}, M.~{Maggioni}, and J.~M. {Murphy}}, {\em {Path-Based Spectral
  Clustering: Guarantees, Robustness to Outliers, and Fast Algorithms}}, arXiv
  e-prints,  (2017), p.~arXiv:1712.06206.

\bibitem{Rohe}
{\sc K.~Rohe, S.~Chatterjee, and B.~Yu}, {\em Spectral clustering and the
  high-dimensional stochastic blockmodel}, Ann. Statist., 39 (2011),
  pp.~1878--1915.

\bibitem{BinYu}
{\sc G.~Schiebinger, M.~J. Wainwright, and B.~Yu}, {\em The geometry of
  kernelized spectral clustering}, Ann. Statist., 43 (2015), pp.~819--846.

\bibitem{art:ShiMalik00NCut}
{\sc J.~Shi and J.~Malik}, {\em Normalized cuts and image segmentation}, IEEE
  Trans. Pattern Anal. Mach. Intell., 22 (2000), pp.~888--905.

\bibitem{Shi2015}
{\sc Z.~{Shi}}, {\em {Convergence of Laplacian spectra from random samples}},
  arXiv e-prints,  (2015), p.~arXiv:1507.00151.

\bibitem{SlepcevThorpe}
{\sc D.~Slep\v{c}ev and M.~Thorpe}, {\em Analysis of p-laplacian regularization
  in semi-supervised learning}, To appear in SIMA,  (2018).

\bibitem{kNNNGT}
{\sc N.~G. Trillos}, {\em Variational limits of $k$-nn graph-based functionals
  on data clouds}.
\newblock To appear in SIMODS, 2019.

\bibitem{Villani}
{\sc C.~Villani}, {\em Topics in optimal transportation}, vol.~58 of Graduate
  Studies in Mathematics, American Mathematical Society, Providence, RI, 2003.

\bibitem{vonLux_tutorial}
{\sc U.~von Luxburg}, {\em A tutorial on spectral clustering}, Stat. Comput.,
  17 (2007), pp.~395--416.

\bibitem{vonLuxburg}
{\sc U.~von Luxburg, M.~Belkin, and O.~Bousquet}, {\em Consistency of spectral
  clustering}, Ann. Statist., 36 (2008), pp.~555--586.

\end{thebibliography}
\bibliographystyle{siam}

\appendix

\section{Supplementary results}

\subsection{Near-orthogonal vectors}

\begin{lemma}\label{BH-finite-dim-orht-basis}
Let $V$ be a vector space of dimension $N$ and let $\langle \cdot, \cdot \rangle$ be an inner product on $V$ with associated norm $\|\cdot\|$. 
Suppose that $v_1, \dots, v_N $ are linearly independent unit vectors in $V$ such that
\[  \lvert \langle  v_j , v_l  \rangle \rvert  \leq \delta, \quad \forall  j \not =l \]
for $\delta>0$ satisfying
\[  N \delta < 1.\]
Then, there exists an orthonormal basis for $V$, $\{ \tilde{v}_1, \dots \tilde{v}_N \}$, such that for every $j=1,\dots, N$
\begin{equation}
\| v_j - \tilde{v}_j \| \leq \tilde{\phi}(N,\delta),
\label{CloseOrthonormalIneq-BH}
\end{equation}
where 
$$
\tilde{\phi}(N, \delta) := \sqrt{N}\left[\frac{1}{\sqrt{1 - N\delta}} - 1\right].
$$
\end{lemma}

\begin{proof}
Without loss of generality assume $V$ is the Euclidean space $\mathbb{R}^N$
with the usual inner product.
Define the square matrix 
$$
\mathbf{V}=
\begin{bmatrix}
v_1 | v_2| \cdots | v_N
\end{bmatrix}.
$$
and the residual matrix $\mathbf{R} = \mathbf{V}^T\mathbf{V} - I$. Since $v_1,...v_N$ are linearly independent, the matrix $\mathbf{V}^T \mathbf{V}$ is invertible, and
the nearest orthonormal approximation to $\mathbf{V}$ in the Frobeneous norm is the matrix 
\cite{higham-nearness-matrix}
$$
\tilde{\mathbf{V}} = \mathbf{V}(\mathbf{V}^T \mathbf{V})^{-\frac{1}{2}} = \mathbf{V}( I + \mathbf{R})^{-\frac{1}{2}}.
$$ 
Using the series expansion of the square root gives 
$$
\mathbf{V} - \tilde{\mathbf{V}} = \mathbf{V}\left(\frac{1}{2} \mathbf{R} - \frac{3}{8} \mathbf{R}^2 + 
\frac{5}{16} \mathbf{R}^3 - \cdots\right) 
$$
Note that $|\mathbf{R}_{ij}| \le \delta$ following our assumptions on the $v_j$. Furthermore, 
a straightforward calculation shows that $|(\mathbf{R}^k)_{ij}| \le N^{k-1}\delta^k$.  Thus, 
$$
\begin{aligned}
|(\mathbf{V} - \tilde{\mathbf{V}})_{ij}| &\le 
N\left(\sup_{ij} |\mathbf{V}_{ij}|\right)\left(\frac{1}{2} \delta  + \frac{3}{8} N \delta^2 + \frac{5}{16} N^2 \delta^3 + \cdots \right)   \\
& = \left(\sup_{ij} |\mathbf{V}_{ij}|\right) \left[ (1 - N\delta)^{-\frac{1}{2}} - 1 \right]\\
& \le \left[ (1 - N\delta)^{-\frac{1}{2}} - 1 \right] = \frac{\tilde{\phi}(N,\delta)}{\sqrt{N}}.
\end{aligned}
$$
It follows that
$$
\|v_j-\tilde{v}_j\|=\left(\sum_{k=1}^N |(\mathbf{V} - \tilde{\mathbf{V}})_{ij}|^2\right)^{1/2} 
\leq \left(\sum_{k=1}^N \frac{\tilde{\phi}(N,\delta)^2}{N}\right)^{1/2}
= \tilde{\phi}(N,\delta)\,,
$$
which proves inequality \eqref{CloseOrthonormalIneq-BH}.
\end{proof}

\begin{remark}
	\label{App:linearIndep}
With the same notation as in the above lemma, suppose that the unit vectors $v_1, \dots, v_N$ satisfy 
\[ |\langle v_i , v_j \rangle| \leq \delta, \quad \forall i\not = j,\]
where $\delta$ satisfies the slightly stronger condition 
\[2N\delta <1 .\]
We claim that in that case the vectors are linearly independent. Indeed, for the sake of contradiction suppose that they are not. Then we can find numbers $a_1, \dots, a_N$ not all equal to zero, for which
\[ 0 = \sum_{i=1}^N a_i v_i .\]
Normalizing the $a_i$ we can further assume that
\[ \sum_{i=1}^N a_i^2 =1.\]
It follows from Jensen's inequality that
\[ 0 = \left\lVert \sum_{i=1}^N a_i v_i \right\rVert^2= \sum_{i=1}^N a_i^2 + 2 \sum_{i=1}^N \sum_{j \not = i} a_i a_j \langle  v_i , v_j \rangle \geq 1 - 2\delta \left( \sum_{i=1}^N|a_i| \right)^2  \geq 1 - 2 N \delta, \]
 which would contradict the hypothesis on $\delta$.
\end{remark}

\subsection{Cheeger's inequality}\label{sec:Cheeger}

Cheeger's inequality in manifolds was introduced in \cite{Cheeger}. For the convenience of the reader here we present a  proof of Cheeger's inequality for sufficiently
smooth and bounded functions using techniques that are developed in spectral geometry and
spectral graph theory literatures. We omit some technical details and highlight
the specific structure of the operator $\Delta_\rho$ which allows us to deduce the inequality.
As we will see below not every normalization of  Laplacian operator will produce a similar result.
We also note that our result can be generalized to functions in $L^2(\M, \rho)$ via a density argument.
 
 \begin{theorem}\label{thm:Cheeger}
 Let  $u \in L^2(\M, \rho)$  sufficiently smooth and bounded with
\[ \int_\M u \rho dx =0. \]
Then
\[ \frac{\langle \Delta_\rho u , u \rangle_\rho }{\langle u,u \rangle_\rho}  \geq \frac{1}{4} h(\M,\rho)^2,\] 
where
\[ h(\M, \rho):= \min_{A \subseteq \M} \frac{\int_{\partial A \cap \M} \rho(x)dS(x) }{ \min \left\{ \int_{A}\rho dx , \int_{\M \setminus A} \rho dx \right\}} . \]
 \end{theorem}

\begin{proof}[Proof]
We show that for any non-constant function $u \in L^2(\M, \rho)$ sufficiently smooth and bounded with
\[ \int_\M u \rho dx =0, \]
we can find a set $A \subseteq \M$ such that
\[ R(u):= \frac{\langle \Delta_\rho u , u \rangle_\rho }{\langle u,u \rangle_\rho} =  \frac{\int_\M |\nabla u |^2 \rho dx }{\int_\M u^2 \rho dx}  \geq\frac{1}{4}\left( \frac{\int_{\partial A \cap \M } \rho dS(x) }{\min \left\{ \int_{A} \rho dx,  \int_{\M \setminus A}  \rho dx \right\} } \right)^{2} \geq \frac{1}{4} (h(\M,\rho))^2.\]  
To see this, we start by letting $r$ be the smallest number for which 
\[ \int_\M \mathds{1}_{\{ u\leq r \}}\rho dx =1/2, \]
and we define
\[ z(x):= u(x) - r, \quad x \in \M. \]
In other words, $r$ is the median of $u(x)$.  
Let $m, M$ be the infimum and supremum of $z$. Then $m<0$ and $M>0$. Notice that since $z$ and $u$ differ only by a constant we have
\[   \langle  \Delta_{\rho} u, u \rangle_\rho =  \langle  \Delta_{\rho} z, z \rangle_\rho,  \]
and from the fact that $\int_{\M}u \rho dx=0$, we also have
\[ \langle z ,z \rangle_\rho \geq \langle u, u\rangle_\rho    \]
In particular,
\[  R(z) \leq R(u). \]
%Since $R(z)$ is unaffected by scalar multiplication of $z$, without the loss of generality we can assume that
%\[ m^2 + M^2=1. \]

Now, the coarea formula states that for any $L^1(dx)$ function $g: \M \rightarrow \R$ we have
\[  \int_{\M} g(x) |\nabla z(x)|dx = \int_{-\infty}^\infty \left( \int_{\partial A_t \cap \M} g(x) dS(x)  \right) dt,    \]
where $A_t$ is the level set:
\[ A_t := \{ x \in \M \: : \: z(x) \leq t   \}.  \]
Taking $g$ to be the function
\[ g(x) = 2\rho(x)|z(x)|, \quad x \in \M, \]
we deduce that
\begin{align}
\int_{\M} 2|z(x)| |\nabla z(x)| \rho(x)dx 
&= \int_{-\infty}^\infty \left( \int_{\partial A_t \cap \M} \rho(x) dS(x)  \right)2 |t| dt\notag\\
&=\int_{m}^{M} \left( \int_{\partial A_t \cap \M} \rho(x) dS(x)  \right)2 |t| dt  ,   
\label{aux:Cheeger}
\end{align}
where in the first equality we have used the fact that at the boundary $\partial A_t$ the function $z$ is equal to $t$. On the other hand,
\begin{align*}
\begin{split}
&\int_{m}^M  \min \left\{  \int_{A_t} \rho dx, \int_{\M \setminus A_t} \rho  dx \right\}   2|t|dt \\
&\qquad= \int_m^0 \left(\int_{A_t} \rho dx\right) 2 |t| dt + \int_0^M  \left(\int_{\M \setminus A_t} \rho dx \right) 2 |t| dt
\\
&\qquad=-\int_m^0  2\left(\int_{A_t} \rho dx\right)t  dt + \int_0^M  \left(\int_{\M \setminus A_t} \rho dx\right) 2 t dt
\\
&\qquad= \int_{\M} \int_m^0 \mathds{1}_{\{ z(x) \le t\}}(x) (-2t) dt    \rho dx +
\int_{\M} \int_0^M \mathds{1}_{\{ z(x) > t\}}(x) (2t) dt    \rho dx
\\
&\qquad= \int_{\M} \int_{z(x)}^0 \mathds{1}_{\{ z(x) \le 0\}}(x) (-2t) dt    \rho dx +
\int_{\M} \int_0^{z(x)} \mathds{1}_{\{ z(x) > 0\}}(x) (2t) dt    \rho dx
\\
&\qquad=\int_{\M} \mathds{1}_{\{z(x) \leq 0\}} z^2 \rho dx + \int_{\M} \mathds{1}_{\{z(x) \geq 0\}} z^2 \rho dx = \int_{\M}z^2 \rho dx,
\end{split}
\end{align*}
where in the third equality we used the Fubini-Tonelli theorem to switch the order of
integrals. From the previous identity and \eqref{aux:Cheeger} we notice that 
\[  \int_{m}^{M} \left( \int_{\partial A_t \cap \M} \rho(x) dS(x)  \right)2 |t| dt   = K(z)   \int_{m}^M  \min \left\{  \int_{A_t} \rho dx, \int_{\M \setminus A_t} \rho  dx \right\}   2|t|dt   \]
where 
\[ K(z) := \frac{2 \int_{\M}|z(x)||\nabla z(x)|\rho(x)dx  }{\langle z, z\rangle_\rho}.\]
Now, Cauchy Schwartz-inequality shows that
\[ K(z) \leq 2 \sqrt {R(z)}.  \]
This is the point where it is important to have both numerator and denominator in the Raleigh quotient $R$ to be weighted by $\rho$. We notice that with a different weighting we would have not gotten the square root of the Raleigh quotient in this last step.

It follows that
\[ 0 \leq \int_m^M2|t| \left(2 \sqrt{R(z)}  \min \left\{  \int_{A_t} \rho dx, \int_{\M \setminus A_t} \rho  dx \right\}   - \int_{\partial A_t \cap \M} \rho(x) dS(x)  \right)  dt ,  \]
from where we can see that there must exist some $t\in (m, M)$ for which 
\[ \frac{\int_{\partial A_t \cap \M} \rho dS(x)}{\min \left\{  \int_{A_t} \rho dx, \int_{\M \setminus A_t} \rho  dx \right\} } \leq 2 \sqrt{R(z)}.  \]
\nc

 \end{proof}

\end{document}